\documentclass[leqno]{siamltex704}
\usepackage{amsmath}
\usepackage{graphicx}
\usepackage{mathrsfs}
\usepackage{bm}
\usepackage{float}
\usepackage{amsfonts,amssymb}
\usepackage{dsfont}
\usepackage{pifont}
\usepackage{tikz}
\usepackage{wrapfig} % Allows in-line images
\usepackage{hyperref}
\usepackage{multirow}
\usepackage{lineno}
\usepackage{mathtools}
\usepackage{appendix}
\usepackage{subfigure}
\usepackage{color}
\newtheorem{remark}{Remark}[section]
\def\S{{\mathfrak s}}
\def\T{{\mathcal T}}
\def\E{{\mathcal E}}

\def\O{{\mathcal O}}
\def\B{{\mathcal B}}
\def\A{{\mathcal A}}

\def\bn{{\bf n}}

\def\pT{{\partial T}}
\def\3bar{{|\!|\!|}}
\def\bbeta{{\boldsymbol\beta}}
\newtheorem{algorithm}{SWG Algorithm}[section]
\newtheorem{FD-algorithm}{5-Point Finite Difference Algorithm}[section]

\newcommand{\vertiii}[1]{{\left\vert\kern-0.25ex\left\vert\kern-0.25ex\left\vert #1
    \right\vert\kern-0.25ex\right\vert\kern-0.25ex\right\vert}}
\title{A discrete maximum principle for the weak Galerkin finite element method on nonuniform rectangular partitions}

\author{Yujie Liu\thanks{School of Data and Computer Science, Sun Yat-sen University, Guangzhou, 510275, China (liuyujie5@mail.sysu.edu.cn). The research of Liu was partially supported by Guangdong Provincial Natural Science Foundation (No. 2017A030310285), Shandong Provincial natural Science Foundation (No. ZR2016AB15) and Youthful Teacher Foster Plan Of Sun Yat-Sen University (No. 171gpy118),} \and Junping Wang
\thanks{Division of Mathematical Sciences, National Science Foundation, Alexandria, VA 22314 (jwang@nsf.gov). The research of Wang was supported by the NSF IR/D program, while working at National Science Foundation. However, any opinion, finding, and conclusions or recommendations
expressed in this material are those of the author and do not
necessarily reflect the views of the National Science Foundation.}}

\begin{document}

\maketitle
%\linenumbers
\begin{abstract}
This article establishes a discrete maximum principle (DMP) for the approximate solution of convection-diffusion-reaction problems obtained from the weak Galerkin finite element method on nonuniform rectangular partitions. The DMP analysis is based on a simplified formulation of the weak Galerkin involving only the approximating functions defined on the boundary of each element. The simplified weak Galerkin method has a reduced computational complexity over the usual weak Galerkin, and indeed provides a discretization scheme different from the weak Galerkin when the reaction term presents. An application of the simplified weak Galerkin on uniform rectangular partitions yields some $5$- and $7$-point finite difference schemes for the second order elliptic equation. Numerical experiments are presented to verify the discrete maximum principle and the accuracy of the scheme, particularly the finite difference scheme.
\end{abstract}



\pagestyle{myheadings}

%
%\noindent {\bf Mathematics Subject Classification (2010)} 65N15;
%65N30; 41A30

\section{Introduction}
In this paper, we are concerned with the development of a discrete maximum principle for the weak Galerkin finite element approximations of convection-diffusion-reaction problems. For simplicity, we consider the problem of seeking an unknown function $u=u(x)$ satisfying
\begin{eqnarray}
-\nabla\cdot(\alpha\nabla u) + \bbeta\cdot\nabla u + cu&=&f\quad {\rm in}\  \Omega  \label{ellipticbdy}\\
u&=&g\quad {\rm on}\ \partial\Omega \label{ellipticbc}
\end{eqnarray}
where $\Omega$ is a bounded polytopal domain in $\mathbb{R}^d \;(d\ge 2)$ with boundary $\partial\Omega$, $\alpha=\alpha(x)$ is the diffusion coefficient, $\bbeta=\bbeta(x)$ is the convection, and $c=c(x)$ is the reaction coefficient in relevant applications. We assume that $\alpha$ is sufficiently smooth, $\bbeta\in [W^{1,\infty}(\Omega)]^d$, and $c$ is piecewise constant with respect to a partition of the domain. For well-posedness of the problem \eqref{ellipticbdy}-\eqref{ellipticbc}, we assume $f=f(x)\in L^2(\Omega)$, $g=g(x)\in H^{\frac12}(\partial\Omega)$, and
\begin{equation}\label{EQ:positive}
c-\frac12\nabla\cdot\bbeta \ge 0,\qquad \alpha(x) \ge \alpha_0 \qquad \forall x \in \Omega
\end{equation}
for a constant $\alpha_0>0$.

The weak Galerkin (WG) finite element method, recently introduced in \cite{WangYe_2013,wy3655, mwy}, refers to a natural extension of the standard finite element methods \cite{ciarlet-fem, gr} where the differential operators are approximated as discrete distributions or discrete weak derivatives and the element continuity circumvented by properly selected stabilizers.
The method has good flexibility in making use of discontinuous elements while sharing the simple formulation of the classical continuous or conforming finite element methods. Weak Galerkin can be easily implemented on finite element partitions consisting of general polygonal or polyhedral elements through the usual assembling strategy of element stiffness matrices.
The method has gained a lot attention and popularity recently and has been successfully applied to a variety of partial differential equations, see \cite{WangYe_2013,LiWang_2013,MWWeiYZhao_2013,WangWang_2016} and the references therein for more details.

A simplified formulation of the weak Galerkin finite element method has been developed in \cite{LiDanWW, LiuWang_SWG} for second order elliptic equations in conjunction with the study of superconvergence and error estimates. This simplification idea was further applied to the Stokes equation in  \cite{LiuWang_SWG_Stokes_2018} for the development of a superconvergence theory on nonuniform rectangular partitions for both the velocity and the pressure approximations. In the simplified weak Galerkin (SWG), the degrees of freedom associated with the unknowns in the interior of each element are eliminated from the usual weak Galerkin method, yielding a numerical scheme with significantly reduced computational complexity. For pure diffusion equations (i.e., $\bbeta=0$ and $c=0$ in \eqref{ellipticbdy}), the simplified weak Galerkin is equivalent to the usual weak Galerkin in the sense that the numerical approximations on the element boundary are the same. But for the full convection-diffusion-reaction equation, these two methods give different numerical solutions on the element boundary. Like WG, the simplified weak Galerkin preserves the important mass conservation property locally on each element and allows the use of general polygonal partitions.

The convection-diffusion-reaction equation \eqref{ellipticbdy} is known to satisfy the following maximum principle: If $u\in C^2(\Omega)\cap C^1(\bar\Omega)$ is the solution of \eqref{ellipticbdy}-\eqref{ellipticbc} with $\alpha, \bbeta, c \in C^1(\Omega)$ and $f\le 0$ in $\Omega$, then $u$ attains its maximum value on $\partial\Omega$ if $c=0$ and $u$ attains its non-negative maximum value on $\partial\Omega$ if $c\ge 0$ \cite{Evans}. A discrete maximum principle (DMP) refers to a similar statement for numerical solutions of \eqref{ellipticbdy}-\eqref{ellipticbc} on specific grid points.
The discrete maximum principle is of great importance from physical point of views in scientific computing (e.g. helping avoid non-physical numerical solutions). In the last two decades, a great deal of effort has been devoted to the search of DMP-preserving numerical methods, see \cite{BertolazziManzini_2005,ChristliebLiuTangXu_2015,DroniouPotier_2011,DraganescuDupontScott_2004,Mudunuro,Varga_1966}. For isotropic diffusion problems, it was shown by Ciarlet and Raviart \cite{CiarletRaviart_1973} that the $P_1$-conforming finite element solutions satisfy a DMP if all the triangular elements have non-obtuse dihedral angles. This nonobtuse-angle condition was improved in \cite{StrangFix} by a weaker condition (namely, the Delaunay condition) that requires the sum of any pair of angles facing a common interior edge be less than or equal to $\pi$. A more recent work on DMP was developed in \cite{WangZhang_2012}
for $P_1$-conforming finite element approximations of quasi-linear second order elliptic equations by using the de Giorgi approach developed in PDE analysis. For discontinuous Galerkin finite element approximations, some DMP-preserving schemes were developed in \cite{ZhZhShu_2013} for convection-diffusion equations on triangular meshes. In the weak Galerkin context and for anisotropic diffusion problems, a DMP result was established for the lowest order WG solutions without stabilization in \cite{HuangWang_2015}. In \cite{WYZhaiZhang_2018}, the authors developed a DMP theory for the WG numerical solutions on the element boundary under a weak acute angle condition for triangular partitions.
For weak Galerkin on rectangular partitions, a DMP result was reported numerically  in \cite{WYZhaiZhang_2018}, but no theory was developed over there.

The goal of this paper is to establish a discrete maximum principle for the simplified weak Galerkin finite element approximations of \eqref{ellipticbdy}-\eqref{ellipticbc} on nonuniform rectangular partitions.
We show that a DMP is satisfied by the numerical solutions arising from SWG with certain values of the stabilization parameter on rectangular partitions for elements with aspect ratio in $[0.5,2]$. For a better understanding of the SWG, we shall compute its global stiffness matrix on uniform rectangular partitions, and develop a 5- and 7-point finite difference scheme for the model diffusion equation. As a variant of the weak Galerkin scheme, these finite difference methods can be shown to preserve the important mass conservation property locally on each cell. We note that a similar finite difference method has been developed in \cite{LiuWang_SWG_Stokes_2018}) for the Stokes equation.

The paper is organized as follows: In Section \ref{section-SWG-polymesh}, we state the simplified weak Galerkin finite element method for the model problem \eqref{ellipticbdy}-\eqref{ellipticbc}. In Section \ref{sectionDMP}, we establish a comprehensive DMP theory for the SWG approximations.
Section \ref{Section:TI} is devoted to the derivation of a technical inequality useful to the DMP development. In Section \ref{sectionFDSWG}, we devise a finite difference scheme for the diffusion equation on uniform Cartesian grids based on the SWG formulation. Finally, in Section \ref{numerical-experiments}, we report some numerical results for a verification of the discrete maximum principle.

In the rest of the paper, we assume $d=2$ and shall use the standard notations for Sobolev spaces and norms \cite{ciarlet-fem,gr}. For any open set $D\subset\mathbb{R}^{2}$, $\|\cdot\|_{s,D}$ and $(\cdot,\cdot)_{s,D}$ denote the norm and inner-product in the Sobolev space $H^s(D)$ consisting of square integrable partial derivatives up to order $s$. When $s=0$ or $D=\Omega$, we shall drop the corresponding subscripts in the norm and inner-product notation.

\section{Simplified Weak Galerkin on Polymesh}\label{section-SWG-polymesh}
Let $\T_h=\{T\}$ be a shape-regular polygonal partition of the domain $\Omega$.  For $T\in \T_h$, denote by $h_T$ its diameter and by $N$ the number of edges.
The meshsize of $\T_h$ is defined as $h=\max_{T\in\T_h} h_T$.
For each edge $e_i, \ i=1,\ldots, N$, let $M_i$ be the midpoints and $\bn_i$ be the outward normal direction of $e_i$; see Fig. \ref{fig.hexahedron}.

Let $v_b$ be a piecewise constant function defined on the boundary of $T$.
The weak gradient of $v_b$ \cite{wy3655,mwy} is given by
\begin{equation}\label{DefWGpoly}
\nabla_w v_b:=\displaystyle\frac{1}{|T|}\sum_{i=1}^N v_{b,i}|e_i|\bf{n_i},
\end{equation}
where $v_{b,i}=v_b|_{e_i}$, $|e_i|$ is the length of the edge $e_i$, and $|T|$ is the area of $T$. The weak gradient $\nabla_w v_b$ can be seen to satisfy the following equation:
\begin{equation}\label{DefWGpoly-new}
(\nabla_w v_b, \bm{\phi})_T=\langle v_b, \bm{\phi}\cdot\bn\rangle_\pT
\end{equation}
for all constant vector $\bm{\phi}$, where and in what follows of the paper, $\langle\cdot,\cdot\rangle_\pT$ is the notation for the usual inner product in $L^2(\pT)$.

Denote by $W(T)$ the local finite element space consisting of piecewise constant functions on $\pT$. The global finite element space, denoted by $W(\T_h)$, is given by patching all the local elements $W(T)$ through common values on interior edges. Denote by $W_h^0(\T_h)$ the closed subspace of $W_h(\T_h)$ consisting of functions with vanishing boundary values.

Let ${P}_k(T)$ be the space of polynomials of degree $k \ge 0$ on $T$. Each $v_b\in W(T)$ can be extended to $T$ as a linear function $\S(v_b)\in {P}_1 (T)$ by the following equations:
\begin{equation}\label{Def.extension}
\sum_{i=1}^{N}(\S(v_b)(M_i) -v_{b,i})\phi(M_i)|e_i|=0,\quad \forall\; \phi\in {P}_1(T).
\end{equation}
It is not hard to show that $\S(u_b)$ is well defined by \eqref{Def.extension}.

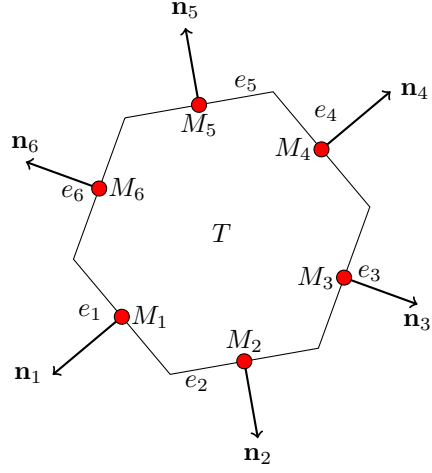
\begin{figure}[!h]
\begin{center}
\begin{tikzpicture}[rotate=40, scale =2.0]

    \path (-0.8660, 0.5) coordinate (A1);
    \path (-0.8660,-0.5) coordinate (A2);
    \path (0.     ,-1  ) coordinate (A3);
    \path (0.8660 ,-0.5) coordinate (A4);
    \path (0.8660 ,0.5 ) coordinate (A5);
    \path (0.     ,1.0 ) coordinate (A6);

    \path (0.0,0.0) coordinate (center);

    \path (-0.8660,0.   )    coordinate (A1half);
    \path (-0.433 ,-0.75)    coordinate (A2half);
    \path ( 0.433 ,-0.75)    coordinate (A3half);
    \path ( 0.8660,0.   )    coordinate (A4half);
    \path ( 0.433 , 0.75)    coordinate (A5half);
    \path (-0.433 , 0.75)    coordinate (A6half);

    \path (A1half) ++(-0.25 ,  0.25)  coordinate (A1halfe);
    \path (A2half) ++(-0.225, -0.0 )  coordinate (A2halfe);
    \path (A3half) ++( 0.225, -0.0 )  coordinate (A3halfe);
    \path (A4half) ++( 0.25 ,  0.25)  coordinate (A4halfe);
    \path (A5half) ++( 0.225,  0.0 )  coordinate (A5halfe);
    \path (A6half) ++(-0.225,  0.0 )  coordinate (A6halfe);

    \path (A1half) ++(-0.6   ,   0   )  coordinate (A1To);
    \path (A2half) ++(-0.2598,  -0.45)  coordinate (A2To);
    \path (A3half) ++( 0.2598,  -0.45)  coordinate (A3To);
    \path (A4half) ++( 0.6   ,   0   )  coordinate (A4To);
    \path (A5half) ++( 0.2598,   0.45)  coordinate (A5To);
    \path (A6half) ++(-0.2598,   0.45)  coordinate (A6To);
    \draw (A6) -- (A1) -- (A2) -- (A3)--(A4)-- (A5)--(A6);

    \draw node at (center) {$T$};
    \draw node[right] at (A1half) {$M_{1}$};
    \draw node[above] at (A2half) {$M_{2}$};
    \draw node[left] at (A3half) {$M_{3}$};
    \draw node[left]  at (A4half) {$M_{4}$};
    \draw node[below] at (A5half) {$M_{5}$};
    \draw node[right] at (A6half) {$M_{6}$};

    \draw node[right] at (A1halfe) {$e_{1}$};
    \draw node[left] at (A2halfe) {$e_{2}$};
    \draw node[below] at (A3halfe) {$e_{3}$};
    \draw node[below]  at (A4halfe) {$e_{4}$};
    \draw node[right] at (A5halfe) {$e_{5}$};
    \draw node[above] at (A6halfe) {$e_{6}$};

    \draw[->,thick] (A1half) -- (A1To) node[left] {$\mathbf{n}_1$};
    \draw[->,thick] (A2half) -- (A2To) node[below]{$\mathbf{n}_2$};
    \draw[->,thick] (A3half) -- (A3To) node[below]{$\mathbf{n}_3$};
    \draw[->,thick] (A4half) -- (A4To) node[right]{$\mathbf{n}_4$};
    \draw[->,thick] (A5half) -- (A5To) node[above]{$\mathbf{n}_5$};
    \draw[->,thick] (A6half) -- (A6To) node[above]{$\mathbf{n}_6$};

    \draw [fill=red] (A1half) circle (0.05cm);
    \draw [fill=red] (A2half) circle (0.05cm);
    \draw [fill=red] (A3half) circle (0.05cm);
    \draw [fill=red] (A4half) circle (0.05cm);
    \draw [fill=red] (A5half) circle (0.05cm);
    \draw [fill=red] (A6half) circle (0.05cm);

\end{tikzpicture}
\end{center}
\caption{An illustrative polygonal element.}
\label{fig.hexahedron}
\end{figure}

On each element $T \in \T_h $, we introduce three bilinear forms:
\begin{eqnarray}
a_T(u_b,v_b)&:= & (\alpha\nabla_w u_b, \nabla_w v_b)_T, \label{EQ:aform}\\
b_T(u_b,v_b)&:= & (\bbeta\cdot\nabla_w u_b, \S(v_b))_T,\label{EQ:bform}\\
c_T(u_b,v_b)&:= & (c\S(u_b), \S(v_b))_T.\label{EQ:cform}
\end{eqnarray}
Furthermore, let
\begin{equation}\label{EQ:Fulla}
\mathcal{B}_T(u_b,v_b):=a_T(u_b,v_b) + b_T(u_b,v_b) + c_T(u_b,v_b)
\end{equation}
for $u_b, v_b \in W(T)$. To enforce a weak continuity, we introduce the following stabilizer:
\begin{equation}\label{EQ:stabilizer}
\begin{split}
S_T(u_b,v_b):= & h^{-1}\sum_{i=1}^N (\S(u_b)(M_i)-u_{b,i})(\S(v_b)(M_i)-v_{b,i})|e_i|\\
             = & h^{-1}\langle Q_b\S(u_b)-u_{b},Q_b\S(v_b)-v_{b}\rangle_{\partial T},
             \end{split}
\end{equation}
where $Q_b$ is the $L^2$ projection operator onto $W(T)$. It is clear that $Q_b u$ is the average of $u$ on each edge. With an abuse of notation, but without confusion, we use $Q_b(g)$ to denote the average of the Dirichlet value $g$ on each boundary edge.
\medskip

\begin{algorithm}
The simplified weak Galerkin (SWG) scheme for the convection-diffusion-reaction equation \eqref{ellipticbdy}-\eqref{ellipticbc} seeks
$u_b \in W_h(\T_h)$ satisfying $u_b = Q_b(g)$ on $\partial\Omega$ and
\begin{equation}\label{equation.SWG}
\A(u_b, v_b) =(f, \S(v_b))\qquad \forall v_b \in W_h^0(\T_h),
\end{equation}
where $\A(\cdot,\cdot)= \kappa S(\cdot,\cdot)+ \B(\cdot, \cdot)$ and
\begin{eqnarray*}
  S(u_b,v_b) = \sum_{T\in\T_h} S_T(u_b,v_b),\quad
  \B(u_b,v_b)= \sum_{T\in\T_h} \mathcal{B}_T(u_b, v_b)
\end{eqnarray*}
are bilinear forms in $W_h(\T_h)$, $(f,\S(v_b)):=\sum_{T\in\T_h} (f, \S(v_b))_T$ is a linear form in $W_h(\T_h)$.
\end{algorithm}

\medskip
The following result on the solution existence and uniqueness has been established by the authors in \cite{LiuWang_SWG}.

\begin{theorem}\label{Lemma:lemma4}
For the model problem \eqref{ellipticbdy}-\eqref{ellipticbc}, assume that $\bbeta\in W^{1,\infty}(\Omega)$ and the ellipticity condition \eqref{EQ:positive} is satisfied. Then, the bilinear form $\A(\cdot,\cdot)$
is bounded and coercive in the finite element space $W_h^0(\T_h)$; i.e., there exist constants $M$ and $\Lambda >0$ such that
\begin{eqnarray}\label{EQ:815:200-00}
|\A(v_b, w_b)| & \le & M \3bar v_b\3bar  \3bar w_b\3bar \qquad \forall v_b, w_b \in W_h^0(\T_h),\\
\label{EQ:815:200-0}
\A(v_b, v_b) & \ge & \Lambda \3bar v_b\3bar^2\qquad \forall v_b\in W_h^0(\T_h),
\end{eqnarray}
provided that the meshsize $h$ of $\T_h$ is sufficiently small. Consequently, the SWG finite element scheme \eqref{equation.SWG} has one and only one solution in the finite element space $W_h(\T_h)$ when the meshsize $h$ is sufficiently small.
\end{theorem}

%Section DMP $$$$$$$$$$$$$$$$$$$$$$$$$$$$$$$$$$$$$$$$$$$$$$$$$$$$$$$$$$$$$$$$$$$$$$$$$$$$$$$$$$$$$$$$$$$$$$$$$$$$$$$$$$$$$$$$$$$$$$$$$$$$$$$$$$$$$$$
\section{Discrete Maximum Principle}\label{sectionDMP}
The model problem \eqref{ellipticbdy}-\eqref{ellipticbc} is known to satisfy the following maximum principle: If $u\in C^2(\Omega)\cap C^1(\bar\Omega)$ is the solution of \eqref{ellipticbdy}-\eqref{ellipticbc} with $\alpha, \bbeta, c \in C^1(\Omega)$ and non-positive $f\in C(\Omega)$, then $u$ attains its maximum value on $\partial\Omega$ when $c=0$ and attains its non-negative maximum value on $\partial\Omega$ when $c\ge 0$ \cite{Evans}. In this section, we show that the maximum principle also holds true for the numerical solutions arising from the SWG scheme \eqref{equation.SWG} on nonuniform rectangular partitions. From now on, the finite element partitions $\T_h$ is assumed to contain only rectangular elements.

%The 7-point finite difference scheme \eqref{equ.scheme-DF-SWG} in Theorem \ref{DF-equi-SWG-new} is one of the examples resulted from the use of mid-point and the Simpson's rule for $(f, \S(v_b))_T$.
%The discrete maximum principles to be developed in this section shall be restricted to certain approximating cases of the SWG scheme \eqref{equation.SWG} arising from some specific numerical integrations, including the one that has led to the finite difference scheme \eqref{equ.scheme-DF-SWG}.

In practical computation, the load linear form $\sum_T(f, \S(v_b))_T$ in the SWG scheme \eqref{equation.SWG} must be approximated by using numerical integrations on each element $T$. Let us first derive a discrete version for $\sum_T(f, \S(v_b))_T$ on which the maximum principles will be established.

Let $T\in\T_h$ be a rectangular element depicted in Fig. \ref{fig:rectangular} with center $M_T=(x_T, y_T)$. Denote by $h_x:=|e_3|=|e_4|$ and $h_y:=|e_1|=|e_2|$ the local meshsize in $x$ and $y$ directions. From \eqref{Def.extension}, the linear extension of $u_b \in W(T)$ can be represented as (see Lemma 6.1 in \cite{LiDanWW} for details)
\begin{equation}\label{EQ:extension}
\S(u_b)=\gamma_0 + \gamma_1(x-x_T)+ \gamma_2(y-y_T),
\end{equation}
where
\begin{equation*}
\left\{
\begin{array}{lllll}
\gamma_0 =\displaystyle\frac{|e_1|(u_{b,1} + u_{b,2}) + |e_3|(u_{b,3} +u_{b,4}) }{2|e_1|+2|e_3|},\\
\gamma_1 =(u_{b,2} - u_{b,1})/|e_3|,\\
\gamma_2 =(u_{b,4} - u_{b,3})/|e_1|.\\
\end{array}
\right.
\end{equation*}
It follows that
\begin{equation}\label{S-Property:002}
\begin{split}
(u_b-{\S}(u_b))(M_1) & =(u_b-{\S}(u_b))(M_2)\\
& = \frac{|e_3|}{2(|e_1|+|e_3|)}(u_{b,1}+u_{b,2}-u_{b,3}-u_{b,4}),\\
\end{split}
\end{equation}
\begin{equation}\label{S-Property:008}
\begin{split}
(u_b-{\S}(u_b))(M_3) & =(u_b-{\S}(u_b))(M_4)\\
& =- \frac{|e_1|}{2(|e_1|+|e_3|)}(u_{b,1}+u_{b,2}-u_{b,3}-u_{b,4}).
\end{split}
\end{equation}

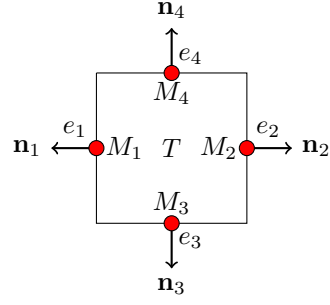
\begin{figure}[!h]
\begin{center}
\begin{tikzpicture}[rotate=0]
    \path (0,0) coordinate (A1);
    \path (2,0) coordinate (A2);
    \path (2,2) coordinate (A3);
    \path (0,2) coordinate (A4);

    \path (1.0,1.0) coordinate (center);

    \path (1.0,0) coordinate (A1half);
    \path (2,1)   coordinate (A2half);
    \path (1.0,2) coordinate (A3half);
    \path (0,1)   coordinate (A4half);

    \path (1.25,0) coordinate (A1halfe);
    \path (2,1.25)  coordinate (A2halfe);
    \path (1.25,2) coordinate (A3halfe);
    \path (0,1.25)  coordinate (A4halfe);

    \path (A1half) ++(0, -0.6)  coordinate (A1To);
    \path (A2half) ++(0.6,  0)  coordinate (A2To);
    \path (A3half) ++(0,  0.6)  coordinate (A3To);
    \path (A4half) ++(-0.6, 0)  coordinate (A4To);
    \draw (A4) -- (A1) -- (A2) -- (A3)--(A4);

    \draw node at (center) {$T$};
    \draw node[above] at (A1half) {$M_{3}$};
    \draw node[left]  at (A2half) {$M_{2}$};
    \draw node[below] at (A3half) {$M_{4}$};
    \draw node[right] at (A4half) {$M_{1}$};

    \draw node[below] at (A1halfe) {$e_{3}$};
    \draw node[right]  at (A2halfe) {$e_{2}$};
    \draw node[above] at (A3halfe) {$e_{4}$};
    \draw node[left] at (A4halfe) {$e_{1}$};

    \draw[->,thick] (A1half) -- (A1To) node[below]{$\mathbf{n}_3$};
    \draw[->,thick] (A2half) -- (A2To) node[right]{$\mathbf{n}_2$};
    \draw[->,thick] (A3half) -- (A3To) node[above]{$\mathbf{n}_4$};
    \draw[->,thick] (A4half) -- (A4To) node[left]{$\mathbf{n}_1$};

    \draw [fill=red] (A1half) circle (0.1cm);
    \draw [fill=red] (A2half) circle (0.1cm);
    \draw [fill=red] (A3half) circle (0.1cm);
    \draw [fill=red] (A4half) circle (0.1cm);
\end{tikzpicture}
\end{center}
\caption{An illustrative rectangular element.}
\label{fig:rectangular}
\end{figure}

Let $\phi_i$ be the local basis function associated with the edge $e_i$ (i.e., $\phi_i=1$ on $e_i$ and $\phi_0=0$ on other edges). For any $v_b=\sum_{i=1}^4 v_{b,i}\phi_i\in W(T)$, we have from \eqref{EQ:extension}
\begin{equation}\label{EQ:MyEQ:200}
\S(v_b) = \sum_{i=1}^4 v_{b,i}\S(\phi_i),
\end{equation}
where
\begin{equation}\label{EQ:MyEQ:201}
\begin{split}
\S(\phi_1) =&\frac{h_y}{2(h_x+h_y)} - \frac{x-x_{T}}{h_x},\quad
\S(\phi_2) =\frac{h_y}{2(h_x+h_y)} + \frac{x-x_{T}}{h_x},\\
\S(\phi_3) =&\frac{h_x}{2(h_x+h_y)} - \frac{y-y_{T}}{h_y},\quad
\S(\phi_4) =\frac{h_x}{2(h_x+h_y)} + \frac{y-y_{T}}{h_y}.\\
\end{split}
\end{equation}
It follows that
\begin{equation}\label{EQ:MyEQ:202}
(f, \S(v_b))_T = \sum_{i=1}^4 v_{b,i} (f, \S(\phi_i))_T.
\end{equation}
Note that $\S(\phi_1)$ is linear in $x$- direction, and constant in $y$-. Thus, the integral $(f, \S(\phi_1))_T$ can be approximated by using the mid-point rule in $y$-direction and the Simpson's rule in $x$-direction:
\begin{equation}\label{EQ:MyEQ:205}
\begin{split}
& (f, \S(\phi_1))_T  = \int_{T} f \S(\phi_1) dxdy\\
= \ &\frac{h_yh_x}{6}\left(f \S(\phi_1)|_{M_1} + 4f\S(\phi_1)|_{(x_T,y_T)} + f\S(\phi_1)|_{M_2}\right) + \O(h^4) \\
= &\frac{h_xh_y}{12(h_x+h_y)}\left((2h_y+h_x)f(M_1) + 4h_y f(x_T, y_T) - h_x f(M_2)\right) + \O(h^4)
\end{split}
\end{equation}
From the Taylor expansion, we have
$$
f(M_1)+f(M_2) = 2 f(x_T,y_T) + \O(h^2).
$$
Substituting the above into \eqref{EQ:MyEQ:205} yields
\begin{equation}\label{EQ:MyEQ:205n}
\begin{split}
& (f, \S(\phi_1))_T\\
= \ & \frac{h_xh_y}{12(h_x+h_y)}\left(2(h_x+h_y)f(M_1) + (4h_y-2h_x)f(x_T, y_T)\right) + \O(h^4)\\
= \ &\frac{|T|}{6} f(M_1) + \frac{|T|(2h_y-h_x)}{6(h_x+h_y)}f(x_T, y_T) +\O(h^4)\\
= \ &\frac{|T|}{6} f(M_1) + \frac{|T|(2-\sigma)}{6(1+\sigma)}f(x_T, y_T) +\O(h^4),
\end{split}
\end{equation}
where $\sigma=h_x/h_y$. Analogously, we have
\begin{eqnarray}
(f, \S(\phi_2))_T  &=&  \frac{|T|}{6} f(M_2) + \frac{|T|(2-\sigma)}{6(1+\sigma)}f(x_T, y_T) + \O(h^4),\label{EQ:MyEQ:206} \\
(f, \S(\phi_3))_T  &=&  \frac{|T|}{6} f(M_3) + \frac{|T|(2-\sigma^{-1})}{6(1+\sigma^{-1})}f(x_T, y_T) + \O(h^4),
\label{EQ:MyEQ:207} \\
(f, \S(\phi_4))_T  &=& \frac{|T|}{6} f(M_4) + \frac{|T|(2-\sigma^{-1})}{6(1+\sigma^{-1})}f(x_T, y_T) + \O(h^4).
\label{EQ:MyEQ:208}
\end{eqnarray}

Using \eqref{EQ:MyEQ:202}-\eqref{EQ:MyEQ:208}, we may approximate $(f,\S(v_b))_T$ by $(f, \S(v_b))_{h,T}$ given as follows:
\begin{equation}\label{EQ:MyEQ:209}
\begin{split}
(f, \S(v_b))_{h,T}:= &\ \frac{|T|}{6}\sum_{i=1}^4 f(M_i)v_{b,i} \\
& \ + \frac{|T|}{6(1+\sigma)}f(x_T,y_T)\left( (2-\sigma)(v_{b,1}+v_{b,2}) + (2\sigma-1)(v_{b,3}+v_{b,4}) \right).
\end{split}
\end{equation}
In the case that $f=f(x,y)$ is non-smooth, the values $f(M_i)$ and $f(x_T,y_T)$ can be substituted by a local average of $f$ around $M_i$ and $(x_T,y_T)$, respectively. Our computational version of the SWG finite element method for the elliptic equation \eqref{ellipticbdy}-\eqref{ellipticbc} can be stated as follows:

\begin{algorithm} Find $u_b \in W_h(\T_h)$ such that $u_b|_{\partial\Omega}=Q_b(g)$ and
\begin{equation}\label{EQ.SWGh}
 \A(u_b,v_b ) =\sum_T(f, \S(v_b))_{h,T}\qquad \forall v_b \in W_h^0(\T_h),
\end{equation}
where $(f, \S(v_b))_{h,T}$ is given by \eqref{EQ:MyEQ:209} on each element $T$.
\end{algorithm}

\medskip
We are now in a position to state the main result of this section on discrete maximum principles.

\begin{theorem}[discrete maximum principle]\label{Th.DMP.squaremesh}
Let $u_b\in W(\T_h)$ be the numerical solution of the model problem \eqref{ellipticbdy}-\eqref{ellipticbc} arising from the scheme (\ref{EQ.SWGh}) on a general rectangular partition $\T_h$ satisfying \eqref{EQ:KappaCond}. Assume the following conditions are satisfied:
\begin{itemize}
\item[(i)] the aspect ratio for each element $T\in \T_h$ satisfies $\sigma\in [0.5, 2]$ with $\sigma=h_y/h_x$,
\item[(ii)] the right-hand side load function is non-positive in $\Omega$ (i.e., $f(x,y)\le 0$ for $(x,y)\in \Omega$).
\end{itemize}
Then the following results hold true:
\begin{itemize}
\item[(a)] If $c=0$ in \eqref{ellipticbdy}, then
\begin{equation}\label{EQ:DMP}
\max_{(x,y)\in\Omega_h}u_b(x,y)\leq \max_{(x,y)\in\partial\Omega_h}Q_b g(x,y).
\end{equation}
\item[(b)] If $c\ge 0$ and has piecewise constant value, then
\begin{equation}\label{EQ:DMP-c}
\max_{(x,y)\in\Omega_h}u_b(x,y)\leq \max_{(x,y)\in\partial\Omega_h}\max(Q_b g(x), 0).
\end{equation}
\end{itemize}
Here in \eqref{EQ:DMP}-\eqref{EQ:DMP-c}, $\Omega_h$ stands for the set of edge midpoints and $\partial\Omega_h$ refers those on the boundary $\partial\Omega$.

%The same inequality holds true for the numerical solution arising from the finite difference schemes \eqref{equ.scheme-DF-SWG} and \eqref{equ.Stokes_FDS} when $0<\kappa \leq 4\alpha$.
\end{theorem}

\begin{proof}
Let $K_M = \displaystyle\max_{(x,y)\in\Omega_h}u_b(x,y)$ and $K^* = \displaystyle\max_{(x,y)\in\partial\Omega_h}Q_bg(x,y)$ if $c\equiv 0$ and $K^* = \displaystyle\max_{(x,y)\in\partial\Omega_h}\max(Q_bg(x,y), 0)$ if $c\ge 0$. We shall show $K_M \leq K^*$ through a contradiction argument. To this end, assume $K^*< K_M$ holds true. For any number $k$ satisfying $K^*<k<K_M$, we define
\begin{equation*}
\theta:=(u_b -k)^+=\left\{
\begin{array}{rl}
u_b - k, &\quad \text{ if } u_b > k\\
0, &\quad \text{ otherwise}
\end{array}
\right.
\end{equation*}
and
\begin{equation*}
\psi:=(u_b -k)^-=\left\{
\begin{array}{rl}
0,&\quad \text{ if } u_b > k\\
u_b - k, &\quad\text{ otherwise}.
\end{array}
\right.
\end{equation*}
It is clear that $\theta + \psi= u_b -k$.
Since $\displaystyle\max_{(x,y)\in\partial \Omega_h}u_b(x,y) =K^*<k$, then $u_b\leq k$ on $\partial \Omega$ so that $\theta=0$ on ${\partial \Omega}$. Furthermore, we have $\theta \neq  0$ on some edges in $\Omega$ as $\max_{(x,y)\in\Omega_h} u_b(x,y)= K_M > k$. Thus, from (\ref{EQ.SWGh}), we have
$$
\A(u_b, v_b) = \sum_T(f, \S(v_b))_{T,h}\qquad \forall v_b\in W_h^0(\T_h).
$$
In particular, by letting $v_b=\theta$ we obtain
\begin{equation*}%\label{EQ:MyEQ300:01}
\A(u_b, \theta) = \sum_T(f, \S(\theta))_{T,h},
\end{equation*}
which, together with (\ref{EQ:MyEQ:209}) and the assumption of $\sigma\in [0.5,2]$, yields
\begin{equation}\label{eq.th.dmp.1}
\begin{split}
\A(u_b, \theta) & =
\frac{|T|}{6}\sum_{i=1}^4 f(M_i)\theta_{b,i} \\
& \ + \frac{|T|}{6(1+\sigma)}f(x_T,y_T)\left( (2-\sigma)(\theta_{b,1}+\theta_{b,2}) + (2\sigma-1)(\theta_{b,3}+\theta_{b,4}) \right),\\
& \le 0,
\end{split}
\end{equation}
where we have used the fact that $\theta_{b,i}\ge 0$, $f\le 0$, $2-\sigma\ge 0$, and $2\sigma-1\ge 0$.
%The inequality \eqref{eq.th.dmp.1} also holds true for the finite difference schemes \eqref{equ.scheme-DF-SWG} and \eqref{equ.Stokes_FDS} when $0<\kappa \leq 4\alpha$.
On the other hand, we have
\begin{equation}\label{eq.th.dmp.2}
\begin{split}
\A(u_b, \theta) =& \sum_T \kappa S_T(u_b,\theta) + \sum_T(\alpha \nabla_w u_b, \nabla_w \theta )_T \\
 & + \sum_T (\bbeta\cdot\nabla_w u_b, \S(\theta))_T + \sum_T (c\S(u_b), \S(\theta))_T\\
=&\kappa h^{-1} \sum_T \langle Q_b\S(u_b)-u_{b}, Q_b\S(\theta)-\theta\rangle_\pT + \sum_T(\alpha \nabla_w u_b, \nabla_w \theta)_T\\
&\ + \sum_T (\bbeta\cdot\nabla_w u_b, \S(\theta))_T +\sum_T (c\S(u_b), \S(\theta))_T \\
=&\kappa h^{-1} \sum_T \langle Q_b\S(u_b-k)-(u_{b}-k), Q_b\S(\theta)-\theta\rangle_\pT \\
& \ + \sum_T(\alpha \nabla_w (u_b-k), \nabla_w \theta)_T + \sum_T (\bbeta\cdot\nabla_w (u_b-k), \S(\theta))_T \\
& +\sum_T (c\S(u_b-k), \S(\theta))_T + (ck, \S(\theta))\\
=& \A(u_b -k, \theta) + (ck, \S(\theta)) \\
= & \A(\theta, \theta) + \A(\psi, \theta) + (ck, \S(\theta)).
%\geq & \A(\theta, \theta).
\end{split}
\end{equation}

Assume the following holds true:
\begin{equation}\label{BIG:assumption}
\A(\psi, \theta) \ge 0.
\end{equation}

For the case of $c=0$, we have from \eqref{eq.th.dmp.2} and \eqref{eq.th.dmp.1} that
\begin{equation}\label{EQ:888}
0\ge \A(u_b,\theta) \ge \A(\theta, \theta),
\end{equation}
which, together with the coercivity inequality \eqref{EQ:815:200-0}, leads to  $\theta \equiv 0$ in $\Omega$ -- a contradiction to the fact that $\theta\neq 0$ on some edges in $\Omega$ as a result of the assumption of $K^*>K_M$. This completes the proof of the discrete maximum principle \eqref{EQ:DMP}.

As to the case of $c\ge 0$, we note that $k>K^* \ge 0$ from the selection of $k$. Furthermore, it is not hard to see that, on each rectangular element $T$, we have
$$
(ck, \S(\theta))_T = \sum_{i=1}^4 ck \theta|_{e_i} \int_T \S(\phi_i) dT \ge 0,
$$
as $ck \theta|_{e_i}\ge 0$ and $\int_T\S(\phi_i) dT >0$ from \eqref{EQ:MyEQ:201}. It follows that the inequality \eqref{EQ:888} again holds true so that $\theta\equiv 0$, which contradicts the fact that $\theta\neq 0$ at some edges. The lemma is thus proved completely.
\end{proof}

\section{A Technical Inequality}\label{Section:TI}
The goal of this section is to verify the validity of the assumption
\eqref{BIG:assumption}. From $\theta=(u_b-k)^+$ and $\psi=(u_b-k)^-$, we may rewrite the assumption as follows:
\begin{equation}\label{BIG:assumption-01}
\A((u_b-k)^-, (u_b-k)^+) \ge 0.
\end{equation}

Let $T\in \T_h$ be a rectangular element depicted in Fig. \ref{fig:rectangular}.
For any $v_b\in W(T)$, define $v_b^+$ and $v_b^-$ as follows
\[
v_b^+ =\max(v_b,0),\; v_b^- =\min(v_b,0).
\]
It is easy to see that
\[
v_b =v_b^+ + v_b^-.
\]
The following lemma provides a version of \eqref{BIG:assumption-01} on the element $T$.
\medskip

\begin{lemma}\label{lemma.DMP.1}
Let $T\in\T_h$ be a rectangular element of size $h_x\times h_y$, and $\sigma=h_x/h_y$ be its aspect ratio. Assume that the stabilization parameter $\kappa$ and the finite element partition $\T_h$ satisfy
\begin{equation}\label{EQ:KappaCond}
\min(\alpha\sigma - \frac{\kappa h_x}{2h(1+\sigma)}, \alpha\sigma^{-1} - \frac{\kappa h_x}{2h(1+\sigma)}) \ge C_0 \|\bbeta\|_\infty h + C_1\|c\|_\infty h^2
\end{equation}
for some prescribed constants $C_0>0$ and $C_1>0$. Then for any $v_b\in W(T)$, the following inequality holds true:
\begin{equation}\label{DMP_inequality}
\kappa S_T(v_b^-, v_b^{+}) + \B_T(v_b^-, v_b^{+}) \ge 0.
%\kappa h^{-1}\langle v_b^+ - Q_b\S(v_b^{+}), v_b^{-} - Q_b\S(v_b^{-}) \rangle_{\partial T} +(\alpha\nabla_w v_b^{+}, \nabla_w v_b^{-})_{T}\geq 0.
\end{equation}
As a result, the inequality \eqref{BIG:assumption-01} holds true for sufficiently small $\kappa$ and sufficiently small meshsize $h$.
\end{lemma}

\begin{proof}
Let $T\in\T_h$ be illustrated in Fig. \ref{fig:rectangular}. Denote by $\phi_i$ the local basis function associated with the edge $e_i$; i.e.,
\begin{equation*}
\phi_i=\left\{
\begin{array}{lllll}
1,\qquad \text{ on } e_i,\\
0,\qquad \text{ on } e_j, \, j\neq i.\\
\end{array}
\right.
\end{equation*}
It follows that
\begin{eqnarray*}
v_b^- - \S(v_b^{-}) = \sum^4_{i=1} v_{b,i}^-(\phi_i - \S(\phi_i)) = \sum^4_{i=1} v_{b,i}^-\varphi_i,
\end{eqnarray*}
where, by using \eqref{S-Property:002} and \eqref{S-Property:008}, the function $\varphi_i:=\phi_i-\S(\phi_i)$ can be shown to satisfy
\begin{equation*}
\varphi_{1}= \varphi_{2}=\left\{
\begin{array}{lll}
\displaystyle\frac{h_x}{2(h_x+h_y)}\quad \text{ at mid-points of } e_1 \text{ and } e_2, \\
\displaystyle\frac{-h_y}{2(h_x+h_y)}\quad \text{ at mid-points of } e_3 \text{ and } e_4, \\
\end{array}
\right.
\end{equation*}
and
\begin{equation*}
\varphi_{3}= \varphi_{4}=\left\{
\begin{array}{lllll}
\displaystyle\frac{-h_x}{2(h_x+h_y)} \quad\text{ at mid-points of } e_1 \text{ and } e_2, \\
\displaystyle\frac{h_y}{2(h_x+h_y)} \quad \text{ at mid-points of } e_3 \text{ and } e_4. \\
\end{array}
\right.
\end{equation*}

Thus, with $\eta= \frac{\kappa |T|}{2h(h_x+h_y)}$, the first term on the left-hand side of \eqref{DMP_inequality} can be computed as
\begin{equation}\label{equ.dmp.1}
\begin{split}
 \kappa S_T(v_b^-, v_b^+) = &\ \kappa h^{-1}\langle v_b^- -Q_b\S(v_b^{-}), v_b^{+} - Q_b\S(v_b^{+}) \rangle_{\partial T} \\
 =& \ \kappa h^{-1}\langle v_b^- -Q_b\S(v_b^{-}), v_b^{+} \rangle_{\partial T}\\
 =& \ \kappa h^{-1}\langle v_b^- -\S(v_b^{-}), v_b^{+} \rangle_{\partial T}\\
 =&\ \kappa h^{-1}\sum_{i=1}^{4}v_{b,i}^- \langle \varphi_i, v_{b}^{+}\rangle_{\partial T}\\
 =&-\eta\left( \sum_{i,j=1}^2 v_{b,i}^-  v_{b,2+j}^{+}  + \sum_{i,j=1}^2 v_{b,i+2}^-  v_{b,j}^{+}\right)\\
&\ +\eta\left(v_{b,1}^-  v_{b,2}^{+}  + v_{b,2}^-  v_{b,1}^{+} + v_{b,3}^-  v_{b,4}^{+} + v_{b,4}^-  v_{b,3}^{+}\right).
\end{split}
\end{equation}
Next, from (\ref{DefWGpoly}) for the weak gradient, we have
\begin{eqnarray}
a_T(v_b^-, v_b^+)&=&(\alpha\nabla_w v_b^{-},\nabla_w v_b^{+})_T \nonumber \\
&=&(|T|^{-1}\alpha\sum_{i=1}^{4}v_{b,i}^{-}|e_i|\bm{n}_i,|T|^{-1}
\sum_{j=1}^{4}v_{b,j}^{+}|e_j|\bm{n}_j)\nonumber \\
&=& \alpha |T|^{-1}\sum_{i,j=1}^4 v_{b,i}^{-} v_{b,j}^{+} |e_i|\ |e_j|\bm{n_i}\cdot\bm{n_j}\label{equ.dmp.2}\\
&=& \alpha |T|^{-1}\sum_{i,j=1, i\neq j}^4 v_{b,i}^{-} v_{b,j}^{+}|e_i|\ |e_j| \bm{n_i}\cdot\bm{n_j} \nonumber\\
& =& -\alpha |T|^{-1} h_y^2 \left(v_{b,1}^-v_{b,2}^+ +v_{b,2}^-v_{b,1}^+ \right)
- \alpha |T|^{-1} h_x^2 \left(v_{b,3}^-v_{b,4}^+ +v_{b,4}^-v_{b,3}^+\right). \nonumber
\end{eqnarray}
By combining (\ref{equ.dmp.1}) with (\ref{equ.dmp.2}) we obtain
\begin{equation}\label{EQ:816:21}
\begin{split}
%&\kappa h^{-1}\langle v_b^+ -S(v_b^{+}), v_b^{-} - S(v_b^{-}) \rangle_{\partial T} + (\alpha\nabla_w v_b^{+},\nabla_w v_b^{-})_T\\
\kappa S_T(v_b^-, v_b^+) + a_T(v_b^-, v_b^+)
= & -\eta\left( \sum_{i=1,j=1}^2 v_{b,i}^-  v_{b,2+j}^{+}  + \sum_{i=1,j=1}^2 v_{b,i+2}^-  v_{b,j}^{+}\right)\\
& - (\alpha |T|^{-1} h_y^2 -\eta) \left(v_{b,1}^-v_{b,2}^+ +v_{b,2}^-v_{b,1}^+ \right) \\
& - (\alpha |T|^{-1} h_x^2 -\eta)\left(v_{b,3}^-v_{b,4}^+ +v_{b,4}^-v_{b,3}^+\right).
\end{split}
\end{equation}

Recall that,  from \eqref{EQ:Fulla}, the bilinear form $\B(\cdot,\cdot)$ is made of three terms:
$$
\B_T(v_b,w_b)=(\alpha\nabla_w v_b, \nabla_w w_b)_T + (\bbeta\cdot\nabla_w v_b, \S(w_b))_T + (c \S(v_b), \S(w_b))_T.
$$
It remains to deal with $(\bbeta\cdot\nabla_w v_b^-, \S(v_b^+))_T$ and $(c \S(v_b^-), \S(v_b^+))_T$. For the former one, we have
\begin{equation}\label{EQ:816:25}
\begin{split}
|b_T(v_b^-, v_b^+)| = & |(\bbeta\cdot\nabla_w v_b^-, \S(v_b^+))_T|\\
= & \left|\sum_{i,j=1, i\neq j}^4 v_{b,i}^- v_{b,j}^+ (\bbeta\cdot \nabla_w\phi_i, \S(\phi_j))_T\right|\\
\leq & C_0\|\bbeta\|_\infty h \left|\sum_{i,j=1, i\neq j}^4 v_{b,i}^- v_{b,j}^+\right|
\end{split}
\end{equation}
for some constant $C_0>0$. As to the later one, we have
\begin{equation}\label{EQ:822:10}
\begin{split}
|c_T(v_b^-, v_b^+)| = & |(c\S(v_b^-), \S(v_b^+))_T|\\
= & \left|\sum_{i,j=1, i\neq j}^4 v_{b,i}^- v_{b,j}^+ (c \S(\phi_i), \S(\phi_j))_T\right|\\
\leq & C_1\|c\|_\infty h^2 \left|\sum_{i,j=1, i\neq j}^4 v_{b,i}^- v_{b,j}^+\right|
\end{split}
\end{equation}
for some constant $C_1$, where we have used the formula \eqref{EQ:MyEQ:201} in the last line of \eqref{EQ:822:10}.

Since $v_{b,i}^{+}v_{b,j}^{-} \leq 0$ for all $1\le i,j \le 4$, then from \eqref{EQ:816:21}, \eqref{EQ:816:25}, and \eqref{EQ:822:10} we obtain
\begin{equation}\label{EQ:816:28}
\begin{split}
%&\kappa h^{-1}\langle v_b^+ -S(v_b^{+}), v_b^{-} - S(v_b^{-}) \rangle_{\partial T} + (\alpha\nabla_w v_b^{+},\nabla_w v_b^{-})_T\\
&\kappa S_T(v_b^-, v_b^+) + \B_T(v_b^-, v_b^+) \\
= \ & \kappa S_T(v_b^-, v_b^+) + a_T(v_b^-, v_b^+)+b_T(v_b^-, v_b^+)+c_T(v_b^-, v_b^+)\\
\ge\ & \kappa S_T(v_b^-, v_b^+) + a_T(v_b^-, v_b^+) - |b_T(v_b^-, v_b^+)|- |c_T(v_b^-, v_b^+)|\\
\ge\ & (\eta-C_0\|\bbeta\|_\infty h-C_1\|c\|_\infty h^2)\left( \sum_{i=1,j=1}^2 |v_{b,i}^-  v_{b,2+j}^{+}|  + \sum_{i=1,j=1}^2 |v_{b,i+2}^-  v_{b,j}^{+}|\right)\\
& + (\alpha |T|^{-1} h_y^2 -\eta-C_0\|\bbeta\|_\infty h-C_1\|c\|_\infty h^2) \left(|v_{b,1}^-v_{b,2}^+| +|v_{b,2}^-v_{b,1}^+| \right) \\
& + (\alpha |T|^{-1} h_x^2 -\eta- C_0\|\bbeta\|_\infty h-C_1\|c\|_\infty h^2)\left(|v_{b,3}^-v_{b,4}^+| +|v_{b,4}^-v_{b,3}^+|\right)\\
\ge \ & 0,
\end{split}
\end{equation}
provided that $\alpha$, $\kappa$, and the rectangular partition $\T_h$ satisfy the following conditions:
\begin{eqnarray*}
\eta & \ge & C_0\|\bbeta\|_\infty h + C_1\|c\|_\infty h^2,\\
\alpha |T|^{-1} h_y^2 -\eta & \ge & C_0\|\bbeta\|_\infty h+C_1\|c\|_\infty h^2,\\
\alpha |T|^{-1} h_x^2 -\eta & \ge & C_0\|\bbeta\|_\infty h+C_1\|c\|_\infty h^2 .
\end{eqnarray*}
The above inequalities hold true under the assumption \eqref{EQ:KappaCond}. This completes the proof of Lemma \ref{lemma.DMP.1}.
\end{proof}

\begin{remark} For square elements $T$, one has $h_x=h_y$ and $\sigma=1$. Thus, by letting $h=h_x$, the condition \eqref{EQ:KappaCond} becomes to be
$$
0 < \kappa \le 4\alpha - 4C_0\|\bbeta\|_\infty h - 4 C_1\|c\|_\infty h^2,
$$
which is easily satisfied for sufficiently small $h$ provided that $\kappa < 4\alpha$. For general rectangular partitions $\T_h$, one may choose $h=2|T|/(h_x+h_y)$ so that the condition \eqref{EQ:KappaCond} becomes to be
$$
0\leq \kappa \leq 4\alpha \min(\sigma, \sigma^{-1})-4C_0\|\bbeta\|_\infty h - 4 C_1\|c\|_\infty h^2,
$$
which is satisfied for $0\leq \kappa < 4\alpha \min(\sigma, \sigma^{-1})$ and  sufficiently small meshsize $h$.
\end{remark}

%Section SWG FD $$$$$$$$$$$$$$$$$$$$$$$$$$$$$$$$$$$$$$$$$$$$$$$$$$$$$$$$$$$$$$$$$$$$$$$$$$$$$$$$$$$$$$$$$$$$$$$$$$$$$$$$$$$$$$$$$$$$$$$$$$$$$$$$$$$$
\section{SWG on Cartesian Grids}\label{sectionFDSWG}
In this section we consider a special case of the SWG finite element method for the model problem \eqref{ellipticbdy}-\eqref{ellipticbc} associated with Cartesian grids (i.e., rectangular partitions). The goal is to derive a finite difference formulation for the simplified weak Galerkin when applied to \eqref{ellipticbdy}-\eqref{ellipticbc}. For simplicity, assume that $\Omega$ is the unit square domain $\Omega= (0,1)\times(0,1)$, and the coefficients in the PDE are given by $\alpha=1, \ \bbeta=0,\ c=0$.

A uniform square partition of the domain can be constructed as the Cartesian product of two one-dimensional grids:
\begin{eqnarray*}
x_{i+\mu}&=&(i+\mu-0.5)h,\qquad i=0,1,\cdots, n,\quad \mu=0,1/2,\\
y_{j+\mu}&=&(j+\mu-0.5)h, \qquad j=0,1,\cdots, n,\quad \mu=0,1/2,
\end{eqnarray*}
where $n$ is a positive integer and $h=1/n$ is the meshsize. Denote by
$$
T_{ij}:= [x_{i-\frac12}, x_{i+\frac12}]\times [y_{j-\frac12}, y_{j+\frac12}],\qquad i,j=1,\ldots, n
$$
the square element centered at $(x_i,y_j)$ for $i,j=1,\cdots, n$. The collection of all such elements forms a uniform square partition $\T_h$ of the domain. The collection of all the element edges is denoted as $\E_h$.

\subsection{A 7-point finite difference scheme}
%We shall compute and assemble the global stiffness matrix and the load vector for the SWG scheme associated with uniform rectangular partitions $\T_h$. The resulting matrix problem can be seen as a finite difference scheme for the elliptic equation. In particular, this will result in a new 5-point finite difference scheme when the stabilization parameter has the value of $\kappa=4$.

Let $u_{k\ell}$ be the approximation of the unknown function $u$ at the mid-point $(x_k, y_\ell)\in\E_h$ (i.e., the dotted points colored in red in Figure \ref{Fig.DF.location}). It can be seen that a red-dotted grid point $(x_k, y_\ell)$ is located on the boundary if either $k$ or $\ell$ takes the value of $\frac12$ or $n+\frac12$. The following is the main result of this section.
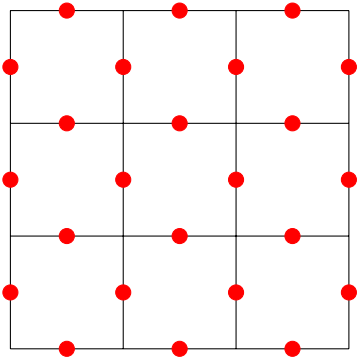
\begin{figure}[!h]
\begin{center}
\begin{tikzpicture}
    \path (0.75,0) coordinate (A11);
    \path (2.25,0) coordinate (A12);
    \path (3.75,0) coordinate (A13);

    \path (0.75,1.5) coordinate (A21);
    \path (2.25,1.5) coordinate (A22);
    \path (3.75,1.5) coordinate (A23);

    \path (0.75,3) coordinate (A31);
    \path (2.25,3) coordinate (A32);
    \path (3.75,3) coordinate (A33);

    \path (0.75,4.5) coordinate (A41);
    \path (2.25,4.5) coordinate (A42);
    \path (3.75,4.5) coordinate (A43);

    \path (0,0.75)   coordinate (B11);
    \path (1.5,0.75) coordinate (B12);
    \path (3,0.75)   coordinate (B13);
    \path (4.5,0.75) coordinate (B14);

    \path (0,2.25)   coordinate (B21);
    \path (1.5,2.25) coordinate (B22);
    \path (3.0,2.25) coordinate (B23);
    \path (4.5,2.25) coordinate (B24);

    \path (0,3.75)   coordinate (B31);
    \path (1.5,3.75) coordinate (B32);
    \path (3.0,3.75) coordinate (B33);
    \path (4.5,3.75) coordinate (B34);

    \draw[step=1.5] (0,0) grid (4.5,4.5);

    \filldraw [fill=red,draw=red] (A11) circle (0.1cm);
    \filldraw [fill=red,draw=red] (A12) circle (0.1cm);
    \filldraw [fill=red,draw=red] (A13) circle (0.1cm);

    \filldraw [fill=red,draw=red] (A21) circle (0.1cm);
    \filldraw [fill=red,draw=red] (A22) circle (0.1cm);
    \filldraw [fill=red,draw=red] (A23) circle (0.1cm);

    \filldraw [fill=red,draw=red] (A31) circle (0.1cm);
    \filldraw [fill=red,draw=red] (A32) circle (0.1cm);
    \filldraw [fill=red,draw=red] (A33) circle (0.1cm);

    \filldraw [fill=red,draw=red] (A41) circle (0.1cm);
    \filldraw [fill=red,draw=red] (A42) circle (0.1cm);
    \filldraw [fill=red,draw=red] (A43) circle (0.1cm);

    \filldraw [fill=red,draw=red] (B11) circle (0.1cm);
    \filldraw [fill=red,draw=red] (B12) circle (0.1cm);
    \filldraw [fill=red,draw=red] (B13) circle (0.1cm);
    \filldraw [fill=red,draw=red] (B14) circle (0.1cm);

    \filldraw [fill=red,draw=red] (B21) circle (0.1cm);
    \filldraw [fill=red,draw=red] (B22) circle (0.1cm);
    \filldraw [fill=red,draw=red] (B23) circle (0.1cm);
    \filldraw [fill=red,draw=red] (B24) circle (0.1cm);

    \filldraw [fill=red,draw=red] (B31) circle (0.1cm);
    \filldraw [fill=red,draw=red] (B32) circle (0.1cm);
    \filldraw [fill=red,draw=red] (B33) circle (0.1cm);
    \filldraw [fill=red,draw=red] (B34) circle (0.1cm);
\end{tikzpicture}
%}
\caption{\label{Fig.DF.location}Grid points of the finite difference scheme.}
\end{center}
\end{figure}

\begin{theorem}\label{DF-equi-SWG-new}
The simplified weak Galerkin scheme \eqref{equation.SWG} on the uniform square partition $\{T_{ij}\}$ for the model problem \eqref{ellipticbdy}-\eqref{ellipticbc} with $\alpha=1, \ \bbeta=0,$ and $c=0$ is algebraically equivalent to the following finite difference scheme: (1) $\{u_{i,j}\}|_{\partial\Omega_h} = Q_b(g)$, and (2) the following equation is satisfied at each interior node:
\begin{equation}\label{equ.scheme-DF-SWG}
\left\{
\begin{split}
&c_1 u_{i+\frac{3}{2},j}+ c_2 u_{i+\frac{1}{2},j}+c_3 u_{i-\frac{1}{2},j} + c_4(u_{i+1,j-\frac{1}{2}}+u_{i+1,j+\frac{1}{2}} +u_{i,j-\frac{1}{2}}+u_{i,j+\frac{1}{2}} ) \\
&\ =\displaystyle\frac{h^2}{2}f_{i+\frac{1}{2},j},\\
&c_1 u_{i,j+\frac{3}{2}}+ c_2 u_{i,j+\frac{1}{2}}+c_3 u_{i,j-\frac{1}{2}} + c_4(u_{i-\frac{1}{2},j+1}+u_{i+\frac{1}{2},j+1} +u_{i-\frac{1}{2},j}+u_{i+\frac{1}{2},j}) \\ &\ =\displaystyle\frac{h^2}{2}f_{i,j+\frac{1}{2}}, \\
\end{split}
\right.
\end{equation}
where
$ c_1=c_3=\displaystyle(\frac{\kappa}{4}-1),\; c_2=\displaystyle\frac{\kappa}{2}+2,\; c_4=\displaystyle-\frac{\kappa}{4}$, $\kappa>0$ is the stabilization parameter, $f_{i+\frac{1}{2},j}=f(x_{i+\frac{1}{2}},y_j)$, $f_{i,j+\frac{1}{2}} =f(x_i,y_{j+\frac{1}{2}}) $.
\end{theorem}

The rest of this subsection will be devoted to a detailed derivation of the finite difference scheme \eqref{equ.scheme-DF-SWG}.

Let $v_b$ be the basis function of $W(T)$ corresponding to the edge $e_1$ of $T$ (see Fig. \ref{fig:rectangular}); i.e.,
\begin{equation*}
v_b=\left\{
\begin{array}{lllll}
1,\qquad \text{ on } e_1,\\
0,\qquad \text{ on } e_i, \, i=2,3,4.\\
\end{array}
\right.
\end{equation*}
Since $|e_1|=|e_3|$ for square elements, from \eqref{S-Property:002} and \eqref{S-Property:008} we have
\begin{eqnarray*}
(v_b-\S(v_b))(M_1) &=&(v_b-\S(v_b))(M_2) =\displaystyle\frac{1}{4},\\
(v_b-\S(v_b))(M_3) &=&(v_b-\S(v_b))(M_4) =-\displaystyle\frac{1}{4}.
\end{eqnarray*}
Hence
\begin{equation}\label{square_stab}
\begin{split}
S_T(u_b,v_b)=&h^{-1}\sum_{i=1}^4 (\S(u_b)(M_i)-u_{b,i})(\S(v_b)(M_i)-v_{b,i})|e_i|\\
            =&h^{-1}\sum_{i=1}^4 (u_{b,i}-\S(u_b)(M_i))v_{b,i}|e_i|\\
            =&\displaystyle\frac{1}{4}(u_{b,1}+u_{b,2} -u_{b,3}-u_{b,4}),
            \end{split}
\end{equation}
and
\begin{equation}\label{square_grad}
\begin{split}
(\nabla_w u_b, \nabla_w v_b)_{T} = &\left(\displaystyle\frac{1}{|T|}\sum_{i=1}^4 u_{b,i}|e_i|\bm{n}_i,\displaystyle\frac{1}{|T|}\sum_{i=1}^4 v_{b,i}|e_i|\bm{n}_i\right) \\
=&\left(
\displaystyle\frac{1}{h}\left[
\begin{array}{lllll}
u_{b,2} - u_{b,1} \\
u_{b,4} - u_{b,3} \\
\end{array}
\right],
\displaystyle\frac{1}{h}\left[
\begin{array}{lllll}
0 - 1\\
0 - 0\\
\end{array}
\right]
\right)_T \\
=&-|T|(\frac{1}{h})^2(u_{b,2}-u_{b,1}) \\
=&u_{b,1}-u_{b,2}.
\end{split}
\end{equation}
Equations \eqref{square_stab} and \eqref{square_grad} comprise the discrete scheme corresponding to the basis function $v_b$ on edge $e_1$ of the element $T$:
\begin{equation}\label{Element_Stencil}
\frac{\kappa}{4}(u_{b,1}+u_{b,2} -u_{b,3}-u_{b,4})+ u_{b,1}-u_{b,2} \approx (f, \S(v_b))_T.
\end{equation}
The two local equations \eqref{Element_Stencil} corresponding to the elements $T_{i,j}$ and $T_{i+1,j}$ that share $e_{i+\frac{1}{2},j}$ as a common edge (see Fig. \ref{Fig.sub.1.ldofswgx}) are given by
\begin{eqnarray}
&& \frac{\kappa}{4}(u_{i+\frac{1}{2},j}+u_{i+\frac{3}{2},j} -u_{i+1,j-\frac{1}{2}}-u_{i+1,j+\frac{1}{2}})+ u_{i+\frac{1}{2},j}-u_{i+\frac{3}{2},j}\nonumber\\
&\approx &(f, \S(v_b))_{T_{i+1,j}},\label{equ.ES.1}\\
&&\frac{\kappa}{4}(u_{i+\frac{1}{2},j}+u_{i-\frac{1}{2},j} -u_{i,j-\frac{1}{2}}-u_{i,j+\frac{1}{2}})+ u_{i+\frac{1}{2},j}-u_{i-\frac{1}{2},j}\nonumber\\
&\approx & (f,\S(v_b))_{T_{i,j}}, \label{equ.ES.2}
\end{eqnarray}

\begin{figure}[h!]
\begin{center}
\subfigure[Stencil for $u_{i+\frac{1}{2},j}$ with any $\kappa>0$]{
\label{Fig.sub.1.ldofswgx}
\tikzset{global scale/.style={
         scale=#1,
         every node/.append style ={scale=#1}
         }
}
\begin{tikzpicture}[global scale =0.75]
\draw[step=3] (0,0) grid (6,3);

\path (0,1.5) coordinate (A11);
\path (3,1.5) coordinate (A12);
\path (6,1.5) coordinate (A13);

\path (1.5,0) coordinate (B11);
\path (4.5,0) coordinate (B12);

\path (1.5,3) coordinate (B21);
\path (4.5,3) coordinate (B22);

    \draw [fill=red] (A11) circle (0.1cm);
    \draw [fill=blue] (A12) circle(0.15cm);
    \draw [fill=red] (A13) circle (0.1cm);

    \draw [fill=red] (B11) circle (0.1cm);
    \draw [fill=red] (B12) circle (0.1cm);

    \draw [fill=red] (B21) circle (0.1cm);
    \draw [fill=red] (B22) circle (0.1cm);

    \path (1.5,1.5) coordinate (center1);
    \path (4.5,1.5) coordinate (center2);

    \draw node[left] at (A11) {$i-\frac{1}{2},j$};
    \draw node[above] at (A12) {$u_{i+\frac{1}{2},j}$};
    \draw node[right] at (A13) {$i+\frac{3}{2},j$};

    \draw node[below] at (B11) {$i,j-\frac{1}{2}$};
    \draw node[below] at (B12) {$i+1,j-\frac{1}{2}$};
    \draw node[above] at (B21) {$i,j+\frac{1}{2}$};
    \draw node[above] at (B22) {$i+1,j+\frac{1}{2}$};

    \draw node at (center1) {$T_{i,j}$};
    \draw node at (center2) {$T_{i+1,j}$};

\end{tikzpicture}
}
~~
%~~~~~~~~~~~~~~~~
\subfigure[Stencil for $u_{i,j+\frac{1}{2}}$ with any $\kappa>0$]{
\label{Fig.sub.1.ldofswgy}
\tikzset{global scale/.style={
         scale=#1,
         every node/.append style ={scale=#1}
         }
}
\begin{tikzpicture}[global scale =0.66]
\draw[step=3] (0,0) grid (3,6);

\path (1.5,0) coordinate (A11);
\path (1.5,3) coordinate (A12);
\path (1.5,6) coordinate (A13);

\path (0,1.5) coordinate (B11);
\path (0,4.5) coordinate (B12);

\path (3,1.5) coordinate (B21);
\path (3,4.5) coordinate (B22);

    \draw [fill=red] (A11) circle (0.1cm);
    \draw [fill=blue] (A12) circle (0.15cm);
    \draw [fill=red] (A13) circle (0.1cm);

    \draw [fill=red] (B11) circle (0.1cm);
    \draw [fill=red] (B12) circle (0.1cm);

    \draw [fill=red] (B21) circle (0.1cm);
    \draw [fill=red] (B22) circle (0.1cm);

    \path (1.5,1.5) coordinate (center1);
    \path (1.5,4.5) coordinate (center2);

    \draw node[below]  at (A11) {$i,j-\frac{1}{2}$};
    \draw node[above]  at (A12) {$u_{i,j+\frac{1}{2}}$};
    \draw node[above]  at (A13) {$i,j+\frac{3}{2}$};

    \draw node[left] at (B11)  {$i-\frac{1}{2},j$};
    \draw node[left] at (B12)  {$i-\frac{1}{2},j+1$};
    \draw node[right] at (B21) {$i+\frac{1}{2},j$};
    \draw node[right] at (B22) {$i+\frac{1}{2},j+1$};

    \draw node at (center1) {$T_{i,j}$};
    \draw node at (center2) {$T_{i,j+1}$};

\end{tikzpicture}
}
\end{center}
\caption{\label{Fig.DF.stencils}Stencils for the 7-point finite difference scheme \eqref{equ.scheme-DF-SWG}.}
\end{figure}

Summing up the equations (\ref{equ.ES.1}) and (\ref{equ.ES.2}) yields the following global linear equation corresponding to the degree of freedom $u_{i+\frac{1}{2},j}$:
\begin{equation}\label{equ.DFSWG.1}
\begin{split}
&\frac{\kappa}{4}(u_{i+\frac{3}{2},j}+ 2u_{i+\frac{1}{2},j}+u_{i-\frac{1}{2},j} -u_{i+1,j-\frac{1}{2}}-u_{i+1,j+\frac{1}{2}} -u_{i,j-\frac{1}{2}}-u_{i,j+\frac{1}{2}})\\
&\ + 2u_{i+\frac{1}{2},j}-u_{i-\frac{1}{2},j}-u_{i+\frac{3}{2},j}\\ =&\int_{T_{i,j}\bigcup T_{i+1,j}} f(x,y) \S(v_b)dxdy.
\end{split}
\end{equation}

The right-hand side of \eqref{equ.DFSWG.1} can be approximated by using numerical integrations (the Simpson rule in the $x$- direction and the midpoint rule in the $y$- direction) as follows:
\begin{equation}\label{equ.Simpson}
\begin{split}
&\int_{T_{i,j}\bigcup T_{i+1,j}} f \S(v_b)dT \\ =&\int_{T_{i,j}}f\S(v_{b,(i+\frac{1}{2},j)}) dT+ \int_{T_{i+1,j}}f\S(v_{b,(i+\frac{1}{2},j)})dT \\
=&\frac{h^2}{6}[f \S(v_{b})(M_{i-\frac{1}{2},j}) + 4f\S(v_{b})(M_{i,j}) + f\S(v_{b})(M_{i+\frac{1}{2},j})] \\
&\ +\frac{h^2}{6}[f \S(v_{b})(M_{i+\frac{1}{2},j}) + 4f\S(v_{b})(M_{i+1,j}) + f\S(v_{b})(M_{i+\frac{3}{2},j})]+\O(h^4) \\
=& \frac{h^2}{6}[-\frac{1}{4}f_{i-\frac{1}{2},j} + f_{i,j} + \frac{3}{4}f_{i+\frac{1}{2},j}]+\frac{h^2}{6}[\frac{3}{4}
f_{i+\frac{1}{2},j} + f_{i+1,j} - \frac{1}{4}f_{i+\frac{3}{2},j}]+\O(h^4)\\
=&\frac{h^2}{24}[-f_{i-\frac{1}{2},j} + 4f_{i,j} +6f_{i+\frac{1}{2},j} +4f_{i+1,j}-f_{i+\frac{3}{2},j}] +\O(h^4) \\
= &\frac{h^2}{2}f_{i+\frac{1}{2},j} + \O(h^4),
\end{split}
\end{equation}
where we have used the following formula:
\begin{eqnarray*}%\label{equ.Extension_vb}
\S(v_{b,(i+\frac{1}{2},j)})|_{T_{i,j}} &=&\frac{1}{4} + \frac{1}{h}(x-x_{T_{i,j}}),\\
\S(v_{b,(i+\frac{1}{2},j)})|_{T_{i+1,j}} &=&\frac{1}{4} - \frac{1}{h}(x-x_{T_{i+1,j}}).
\end{eqnarray*}
Thus, we may rewrite \eqref{equ.DFSWG.1} as follows:
\begin{equation}\label{equ.DFSWG.2}
\begin{split}
& \frac{\kappa}{4}(u_{i+\frac{3}{2},j}+ 2u_{i+\frac{1}{2},j}+u_{i-\frac{1}{2},j} -u_{i+1,j-\frac{1}{2}}-u_{i+1,j+\frac{1}{2}} -u_{i,j-\frac{1}{2}}-u_{i,j+\frac{1}{2}})\\
&\quad + 2u_{i+\frac{1}{2},j}-u_{i-\frac{1}{2},j}-u_{i+\frac{3}{2},j}
=\frac{h^2}{2}f_{i+\frac{1}{2},j} + \O(h^4).
\end{split}
\end{equation}
Analogously, for the degree of freedom $u_{i,j+\frac{1}{2}}$, we have
\begin{equation}\label{equ.DFSWG.3}
\begin{split}
& \frac{\kappa}{4}(u_{i,j+\frac{3}{2}}+ 2u_{i,j+\frac{1}{2}}+u_{i,j-\frac{1}{2}} -u_{i-\frac{1}{2},j+1}-u_{i+\frac{1}{2},j+1} -u_{i-\frac{1}{2},j}-u_{i+\frac{1}{2},j})\\
&\quad + 2u_{i,j+\frac{1}{2}}-u_{i,j-\frac{1}{2}}-u_{i,j+\frac{3}{2}}
=\frac{h^2}{2}f_{i,j+\frac{1}{2}}+ \O(h^4).
\end{split}
\end{equation}

The stencil for the unknown $u$, or more precisely for $u_{i+\frac{1}{2},j}$ (respectively $u_{i,j+\frac{1}{2}}$ ) is the seven dotted-points shown in Fig. \ref{Fig.DF.stencils}(a) (respectively Fig. \ref{Fig.DF.stencils}(b)) with weights
\begin{equation}\label{EQ:weight-7points}
A=(\displaystyle\frac{\kappa}{4}-1,\displaystyle\frac{\kappa}{2}+2,
\displaystyle\frac{\kappa}{4}-1,-\displaystyle\frac{\kappa}{4},
-\displaystyle\frac{\kappa}{4},-\displaystyle\frac{\kappa}{4},
-\displaystyle\frac{\kappa}{4})
\end{equation}
for $\kappa>0$.

A similar 7-point finite difference scheme can be derived by following the same calculation for general nonuniform rectangular partitions, for which the stencil should be modified accordingly; details are left to interested readers as an exercise.

\subsection{A 5-point finite difference scheme}\label{subsection-5point}

For the particular value of $\kappa=4$, the weights in \eqref{EQ:weight-7points} for the 7-point stencil become to be
\begin{equation*}\label{EQ:weight-7points-s}
A=(0, 4, 0, -1, -1, -1, -1)
\end{equation*}
so that \eqref{equ.scheme-DF-SWG} is reduced to a 5-point finite difference scheme described as follows:

\begin{FD-algorithm}
On the set of grid points $\Omega_h$, find $\{u_{k\ell}\}$ satisfying (1)
the discrete Dirichlet boundary condition of $\{u_{k\ell}\}|_{\partial\Omega_h}=Q_b g$ at all the red-dotted grid points $(x_k, y_\ell)$ on the domain boundary, and (2) the following set of linear equations:
\begin{equation}\label{equ.Stokes_FDS}
\left\{
\begin{split}
\frac{4u_{i+\frac{1}{2},j}-u_{i+1,j-\frac{1}{2}} - u_{i+1,j+\frac{1}{2}} - u_{i,j-\frac{1}{2}}- u_{i,j+\frac{1}{2}}}{h^2}
&=\frac{1}{2}f_{i+\frac{1}{2},j},\\
\frac{4u_{i,j+\frac{1}{2}}-u_{i-\frac{1}{2},j+1} - u_{i+\frac{1}{2},j+1} - u_{i-\frac{1}{2},j}- u_{i+\frac{1}{2},j}}{h^2}
&=\frac12 f_{i, j+\frac{1}{2}},\\
\end{split}
\right.
\end{equation}
where
$$
f_{k,\ell}:=f(x_k,y_\ell).
$$
\end{FD-algorithm}

\medskip

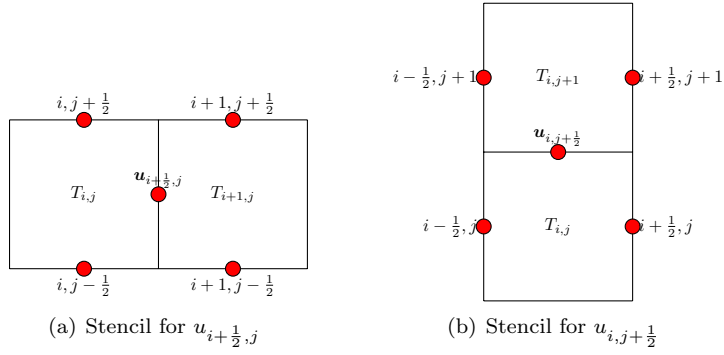
\begin{figure}[!h]
\begin{center}
\subfigure[Stencil for ${u}_{i+\frac{1}{2},j}$]{
\label{Fig.sub.2.ldofswgx}
\tikzset{global scale/.style={
         scale=#1,
         every node/.append style ={scale=#1}
         }
}
\begin{tikzpicture}[global scale =0.66]
\draw[step=3] (0,0) grid (6,3);
\path (0,1.5) coordinate (A11);
\path (3,1.5) coordinate (A12);
\path (6,1.5) coordinate (A13);

\path (1.5,0) coordinate (B11);
\path (4.5,0) coordinate (B12);

\path (1.5,3) coordinate (B21);
\path (4.5,3) coordinate (B22);

    \draw [fill=red] (A12) circle(0.15cm);

    \draw [fill=red] (B11) circle (0.15cm);
    \draw [fill=red] (B12) circle (0.15cm);

    \draw [fill=red] (B21) circle (0.15cm);
    \draw [fill=red] (B22) circle (0.15cm);

    \path (1.5,1.5) coordinate (center1);
    \path (4.5,1.5) coordinate (center2);

    \draw node[above] at (A12) {$\bm{u}_{i+\frac{1}{2},j}$};

    \draw node[below] at (B11) {$i,j-\frac{1}{2}$};
    \draw node[below] at (B12) {$i+1,j-\frac{1}{2}$};
    \draw node[above] at (B21) {$i,j+\frac{1}{2}$};
    \draw node[above] at (B22) {$i+1,j+\frac{1}{2}$};

    \draw node at (center1) {$T_{i,j}$};
    \draw node at (center2) {$T_{i+1,j}$};

\end{tikzpicture}
}
~~~~
\subfigure[Stencil for ${u}_{i,j+\frac{1}{2}}$]{
\label{Fig.sub.2.ldofswgy}
\tikzset{global scale/.style={
         scale=#1,
         every node/.append style ={scale=#1}
         }
}
\begin{tikzpicture}[global scale =0.66]
\draw[step=3] (0,0) grid (3,6);

\path (1.5,0) coordinate (A11);
\path (1.5,3) coordinate (A12);
\path (1.5,6) coordinate (A13);

\path (0,1.5) coordinate (B11);
\path (0,4.5) coordinate (B12);

\path (3,1.5) coordinate (B21);
\path (3,4.5) coordinate (B22);

    \draw [fill=red] (A12) circle(0.15cm);

    \draw [fill=red] (B11) circle (0.15cm);
    \draw [fill=red] (B12) circle (0.15cm);

    \draw [fill=red] (B21) circle (0.15cm);
    \draw [fill=red] (B22) circle (0.15cm);

    \path (1.5,1.5) coordinate (center1);
    \path (1.5,4.5) coordinate (center2);

    \draw node[above]  at (A12) {$\bm{u}_{i,j+\frac{1}{2}}$};

    \draw node[left] at (B11)  {$i-\frac{1}{2},j$};
    \draw node[left] at (B12)  {$i-\frac{1}{2},j+1$};
    \draw node[right] at (B21) {$i+\frac{1}{2},j$};
    \draw node[right] at (B22) {$i+\frac{1}{2},j+1$};

    \draw node at (center1) {$T_{i,j}$};
    \draw node at (center2) {$T_{i,j+1}$};

\end{tikzpicture}
}
\end{center}
\caption{\label{Fig.DF.stencil2}Stencils for the 5-point finite difference scheme \eqref{equ.Stokes_FDS}.}
\end{figure}

Figure \ref{Fig.DF.stencil2} shows the stencil of  \eqref{equ.Stokes_FDS} as a five-point finite difference scheme with weights $\tilde A=(4,-1,-1,-1,-1)$; the left figure is for the scheme centered at $(x_{i+\frac12,j}, y_j)$ and the right one is for $(x_{i}, y_{j+\frac12})$.

Due to the connection between the finite difference scheme \eqref{equ.scheme-DF-SWG} and the SWG finite element method, we have the following result for the solvability of the finite difference method.
\begin{theorem}\label{DF_uniq}
For any given stabilization parameter $\kappa>0$, the finite difference scheme \eqref{equ.scheme-DF-SWG} has one and only one solution. The same conclusion holds true for the five-point finite difference scheme \eqref{equ.Stokes_FDS}. In addition, the discrete maximum principle Theorem \ref{Th.DMP.squaremesh} holds true for the numerical solutions obtained from the finite difference schemes.
\end{theorem}

\section{Numerical Experiments}\label{numerical-experiments}
In this section, we shall numerically verify the discrete maximum principle (i.e.,  Theorem \ref{Th.DMP.squaremesh}) for the simplified weak Galerkin finite element method on rectangular partitions. In addition, some numerical results will be reported to illustrate the convergence and the accuracy of the finite difference scheme \ref{equ.Stokes_FDS}.

The following norms are used to compute the error of the numerical solutions:
\begin{eqnarray*}
&&\text{Discrete $L^2$-norm: }\\
&&\|u_b-u \|_{0,h}=h\left(\sum_{i=1}^{n+1}\sum_{j=1}^{n}|u_{i-\frac{1}{2},j}-u(x_{i-\frac{1}{2}},y_j)|^2  + \sum_{i=1}^{n}\sum_{j=1}^{n+1}|u_{i,j-\frac{1}{2}}-u(x_{i},y_{j-\frac{1}{2}})|^2\right)^{1/2},\\
&&\text{Discrete $H^1$-norm: }\\
&&\|u_b-u\|_{1,h}= h \left(\sum_{i=1}^{n}\sum_{j=1}^{n} \left|\frac{u_{i+\frac{1}{2},j}-u_{i-\frac{1}{2},j}}{h} - \frac{\partial u}{\partial x} (x_{i},y_{j})\right|^2 \right. \\
&&\qquad\qquad\qquad \left. + \sum_{i=1}^{n}\sum_{j=1}^{n} \left|\frac{u_{i,j+\frac{1}{2}}-u_{i,j-\frac{1}{2}}}{h} - \frac{\partial u}{\partial y} (x_{i},y_{j})\right|^2  \right)^{1/2} ,\\
\end{eqnarray*}

\subsection{Verification of the discrete maximum principle}
Three test cases are considered in the verification of DMP for the model problem
\eqref{ellipticbdy}-\eqref{ellipticbc} on rectangular domains $\Omega$. In all the numerical tests, the right-hand side function $f$ is non-positive in $\Omega$. The first test problems assume a vanishing reaction coefficient (i.e., $c=0$), while the third one assumes a positive constant for the reaction coefficient.

{\bf Test Case 1:}\
In this test, the model problem \eqref{ellipticbdy}-\eqref{ellipticbc} is defined on $\Omega=(0,1)^2$ with solutions and coefficients given by
\begin{equation}\label{eq.testcase.1}
\left\{
\begin{split}
&u=x^2+2xy,  \\
&\alpha =
\begin{bmatrix}
 1 &0\\
 0 & 1
\end{bmatrix},
\quad \bm{\beta}=
\begin{bmatrix}
-1\\
-1
\end{bmatrix},\quad c=0,\\
&f=-2-4x-2y.
\end{split}\right.
\end{equation}
The Dirichlet boundary data $g$ is chosen to match the exact solution. Uniform square partitions (i.e. $\sigma=1$) are employed in the SWG scheme \eqref{equation.SWG} with three values for the stabilization parameter $\kappa=0.7$, $4.0$, $20.0$. Table \ref{tab.testcase1} shows the corresponding error and convergence information for the numerical solutions. The numerical results illustrate an optimal order of convergence in the $L^2$ norm. In addition, a superconvergence of order $\O(h^2)$ was observed in the discrete $H^1$ norm when $\kappa=0.7$ and $\kappa=20.0$. For the case of $\kappa=4.0$, the numerical approximations are of machine accuracy so that no rate of convergence is necessary.

\begin{table}[H]
\begin{center}
{\small
\caption{Error and convergence performance of the SWG scheme \eqref{equation.SWG} for the test problem \eqref{eq.testcase.1} on uniform square partitions.}\label{tab.testcase1}
\begin{tabular}{|c|cc|cc|cc|cc|}
\hline
     & \multicolumn{4}{| c |}{$\kappa=0.7$}& \multicolumn{4}{| c |}{$\kappa=4.0$}\\
\hline
$h^{-1}$ & $\|u_h-u \|_{0,h}$ & $r=$ & $\|u_h-u\|_{1,h}$ & $r=$& $\|u_h-u \|_{0,h}$ & $r=$ & $\|u_h-u\|_{1,h}$ & $r=$\\
\hline
8 &    1.79e-02  & 1.6   &  5.81e-02 &  1.5 & 1.71e-15 &     -  &   6.64e-15 &     - \\
16&    4.94e-03  & 1.9   &  1.72e-02 &  1.8 & 7.11e-16 &     -  &   5.77e-15 &     - \\
32&    1.27e-03  & 2.0   &  4.83e-03 &  1.8 & 3.51e-15 &     -  &   1.78e-14 &     - \\
64&    3.22e-04  & 2.0   &  1.32e-03 &  1.9 & 7.78e-15 &     -  &   5.08e-14 &     - \\
\hline
$h^{-1}$ & \multicolumn{4}{| c |}{$\kappa=20.$}& \multicolumn{4}{c}{\multirow{6}{*}{}}\\
\cline{1-5}
& $\|u_h-u \|_{0,h}$ & $r=$ & $\|u_h-u\|_{1,h}$ & $r=$ \\
\cline{1-5}
8 &    3.57e-03 &  2.0  &   1.37e-02  & 1.9\\
16&    8.81e-04 &  2.0  &   3.72e-03  & 1.9\\
32&    2.19e-04 &  2.0  &   9.99e-04  & 1.9\\
64&    5.48e-05 &  2.0  &   2.66e-04  & 1.9\\
\cline{1-5}
\end{tabular}
}
\end{center}
\end{table}

Figure \ref{fig.DMP1} shows the surface plot of the SWG approximation and its maximum value on the domain boundary for the test problem \eqref{eq.testcase.1}. Each of the surface plot indicates that the maximum value of the SWG solution is attained on the domain boundary, which is consistent with the DMP theory developed in Theorem \ref{Th.DMP.squaremesh}). We note that the DMP theory was not applicable to the case of $\kappa=20$, though it is numerically valid.

\begin{figure}[h!]
\centering
\subfigure[Numerical solutions and their maximum values on the domain boundary: rectangular partitions of size $4\times4$ (left), $16\times16$ (middle) and $64\times 64$ (right) with $\kappa=0.7$.]{
\label{Fig.sub1.1}
\includegraphics [width=0.32\textwidth]{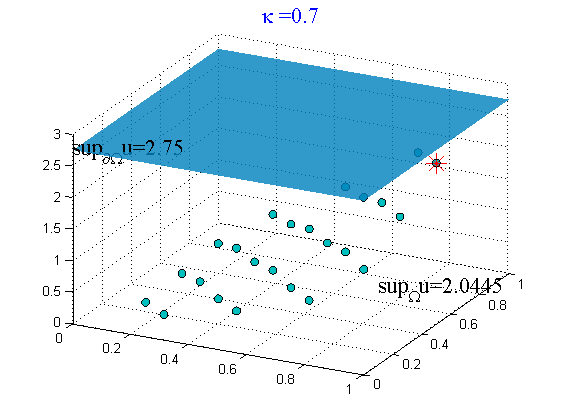}

\includegraphics [width=0.32\textwidth]{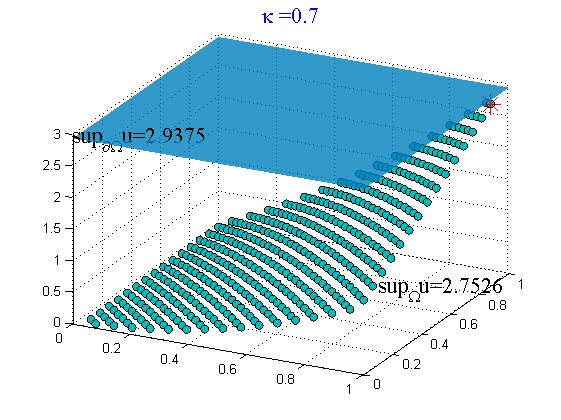}

\includegraphics [width=0.32\textwidth]{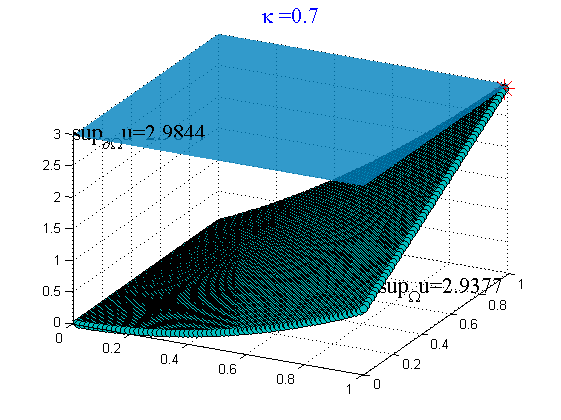}
}\\
\subfigure[Numerical solutions and their maximum values on the domain boundary: rectangular partitions of size $4\times4$ (left), $16\times16$ (middle) and $64\times 64$ (right) with $\kappa=4.0$.]{
\label{Fig.sub1.2}
\includegraphics [width=0.32\textwidth]{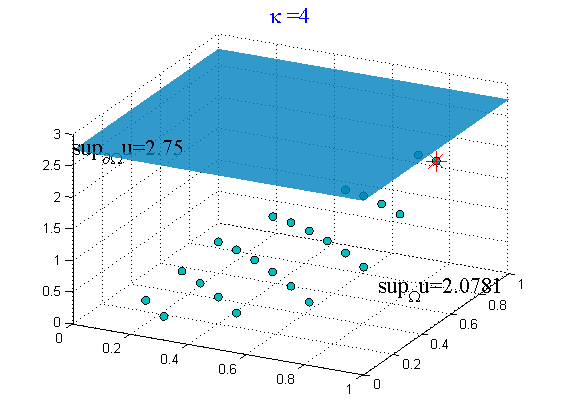}

\includegraphics [width=0.32\textwidth]{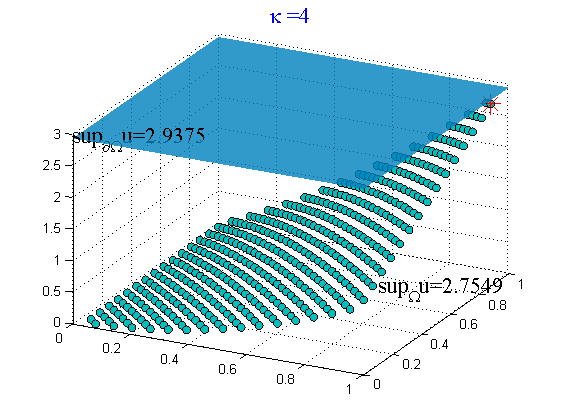}

\includegraphics [width=0.32\textwidth]{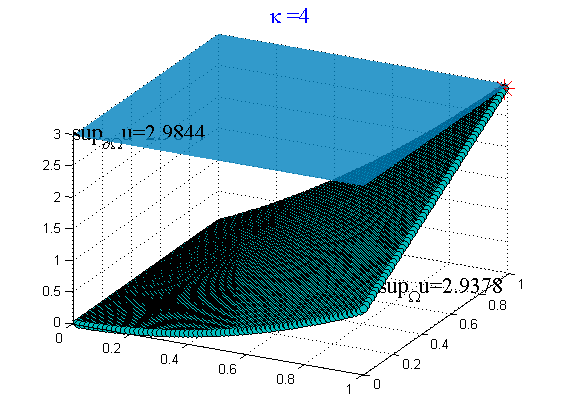}
}\\
\subfigure[Numerical solutions and their maximum values on the domain boundary: rectangular partitions of size $4\times4$ (left), $16\times16$ (middle) and $64\times 64$ (right) with $\kappa=20.0$.]{
\label{Fig.sub1.3}
\includegraphics [width=0.32\textwidth]{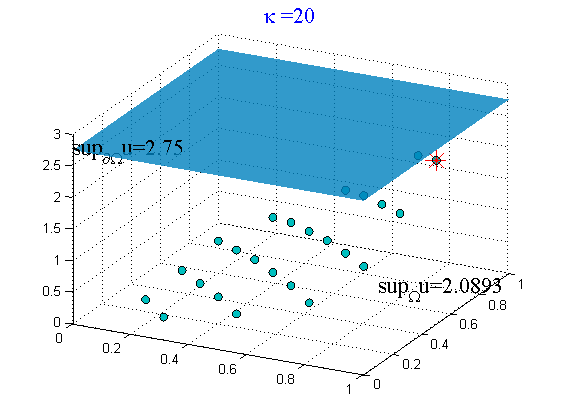}

\includegraphics [width=0.32\textwidth]{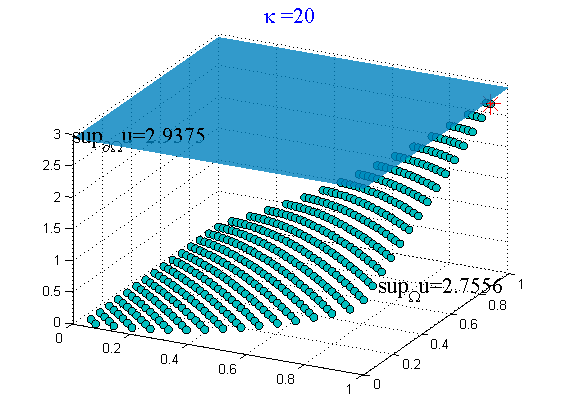}

\includegraphics [width=0.32\textwidth]{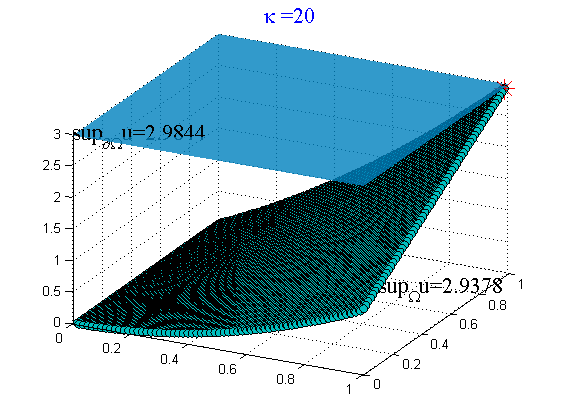}
}\\
\caption{\label{fig.DMP1} DMP verification of the SWG scheme \eqref{equation.SWG} for the test problem \eqref{eq.testcase.1}.}
\end{figure}

{\bf Test Case 2:}\
In this test, the model problem \eqref{ellipticbdy}-\eqref{ellipticbc} has the following configuration on its solution and the coefficients:
\begin{equation}\label{eq.testcase.2}
\left\{
\begin{split}
&u=-\sin(x)\sin(y), \\
&\alpha=
\begin{bmatrix}
 xy+1 &0\\
 0    & 3xy
\end{bmatrix},
\quad
\bm{\beta}=
\begin{bmatrix}
y\\
3x
\end{bmatrix},
\quad c=0, \\
& f= -(4xy+1) \sin(x)\sin(y).
\end{split}\right.
\end{equation}
The domain is the unit square, and the Dirichlet boundary value $g$ is chosen to match the exact solution.

Table \ref{tab.testcase2} illustrates the error and convergence performance of the numerical solutions arising from the SWG scheme \eqref{equation.SWG} on uniform square partitions with three values of $\kappa=0.7$, $4.0$, $20.0$.
The table shows an optimal order of convergence (i.e., $\O(h^2)$) in the $L^2$ norm. The computational results from this test also suggest a superconvergence of the numerical solution in the discrete $H^1$ norm for all three values of $\kappa$.

\begin{table}[ht]
{\small
\begin{center}
\caption{Error and convergence performance of the SWG scheme \eqref{equation.SWG} for the test problem \eqref{eq.testcase.2} on uniform square partitions.}\label{tab.testcase2}
\begin{tabular}{|c|cc|cc|cc|cc|}
\hline
     & \multicolumn{4}{| c |}{$\kappa=0.7$}& \multicolumn{4}{| c |}{$\kappa=4.0$}\\
\hline
$h^{-1}$ & $\|u_h-u \|_{0,h}$ & $r=$ & $\|u_h-u\|_{1,h}$ & $r=$& $\|u_h-u \|_{0,h}$ & $r=$ & $\|u_h-u\|_{1,h}$ & $r=$\\
\hline
 8  &  1.14e-02 &  1.5  & 4.58e-02  & 1.2 &  2.34e-03 & 1.9  & 1.03e-02 &  1.6\\
 16 &  3.17e-03 &  1.9  & 1.54e-02  & 1.6 &  5.94e-04 & 2.0  & 3.17e-03 &  1.7\\
 32 &  8.15e-04 &  2.0  & 4.85e-03  & 1.7 &  1.49e-04 & 2.0  & 9.62e-04 &  1.7\\
 64 &  2.05e-04 &  2.0  & 1.50e-03  & 1.7 &  3.73e-05 & 2.0  & 2.91e-04 &  1.7\\
\hline
$h^{-1}$ & \multicolumn{4}{| c |}{$\kappa=20.$}& \multicolumn{4}{c}{\multirow{6}{*}{}}\\
\cline{1-5}
& $\|u_h-u \|_{0,h}$ & $r=$ & $\|u_h-u\|_{1,h}$ & $r=$ \\
\cline{1-5}
  8  & 6.42e-04 &  2.0  &   2.83e-03 &  1.7\\
  16 & 1.63e-04 &  2.0  &   8.27e-04 &  1.8\\
  32 & 4.12e-05 &  2.0  &   2.40e-04 &  1.8\\
  64 & 1.04e-05 &  2.0  &   6.96e-05 &  1.8\\
\cline{1-5}
\end{tabular}
\end{center}
}
\end{table}

Figure \ref{fig.DMP2} shows the numerical solutions and their maximum values on the domain boundary for the test problem \eqref{eq.testcase.2} on the $16\times 16$ uniform square partition. It can be seen that DMP is satisfied for each of the test value of $\kappa$. Similar tests were conducted on finer rectangular partitions such as $32\times32$ and $64\times64$, and the DMP was observed in all of the computations.

\begin{figure}[h!]
\includegraphics [width=0.32\textwidth]{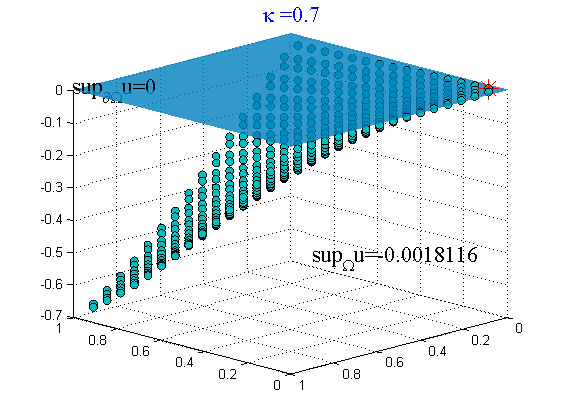}
\includegraphics [width=0.32\textwidth]{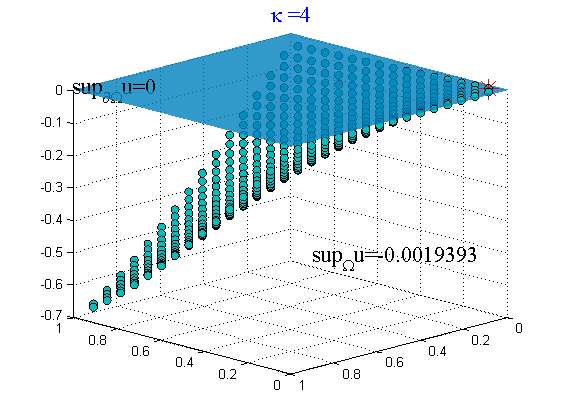}
\includegraphics [width=0.32\textwidth]{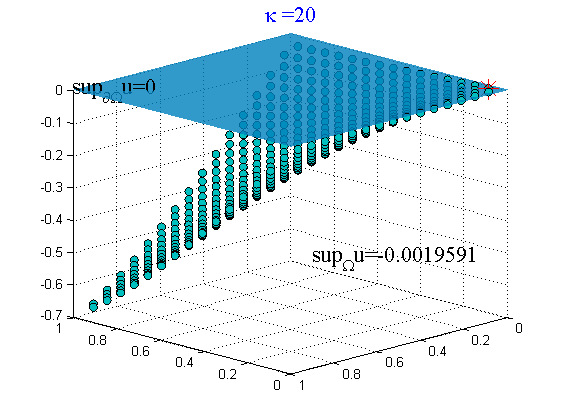}
\caption{\label{fig.DMP2} DMP verification of the SWG scheme \eqref{equation.SWG} for the test problem \eqref{eq.testcase.2} using the partition of size $16\times16$: $\kappa=0.7$ (left), $\kappa=4.0$ (middle), and $\kappa=20.0$ (right).}
\end{figure}

\medskip
{\bf Test Case 3:}\
The model problem \eqref{ellipticbdy}-\eqref{ellipticbc} is now defined in $\Omega=(-1,1)^2$. The solution and the PDE coefficients are given by
\begin{equation}\label{eq.testcase.3}\left\{
\begin{split}
&u=-(x^2(x^2-1.2)-0.3)(y^2(y^2-1.2)-0.3),\\
&\alpha=
\begin{bmatrix}
 1 &0\\
 0    & 1
\end{bmatrix},
\quad
\bm{\beta}=
\begin{bmatrix}
0\\
0
\end{bmatrix},
\quad c=16, \\
&f= 8y^2(2.7-y^2)(x^2(x^2-1.2)-0.3) + 8x^2(2.7-x^2)(y^2(y^2-1.2)-0.3).
\end{split}\right.
\end{equation}
The Dirichlet boundary data $g$ is computed to match the exact solution $u=u(x,y)$.

Table \ref{tab.testcase3} shows the error and convergence performance of the numerical solutions on uniform square partitions with three values of $\kappa=0.7$, $4.0$, $20.0$. Optimal order of convergence (i.e., $r=2$) is clearly indicated in the table for the $L^2$ error. In addition, a superconvergence of order $\O(h^2)$ is observed in the discrete $H^1$ norm for all three test cases of $\kappa$.

\begin{table}[ht]
{\small
\begin{center}
\caption{Error and convergence performance of the SWG scheme \eqref{equation.SWG} for the test problem \eqref{eq.testcase.3} on uniform square partitions.}\label{tab.testcase3}
\begin{tabular}{|c|cc|cc|cc|cc|}
\hline
     & \multicolumn{4}{| c |}{$\kappa=0.7$}& \multicolumn{4}{| c |}{$\kappa=4.0$}\\
\hline
$h^{-1}$ & $\|u_h-u \|_{0,h}$ & $r=$ & $\|u_h-u\|_{1,h}$ & $r=$& $\|u_h-u \|_{0,h}$ & $r=$ & $\|u_h-u\|_{1,h}$ & $r=$\\
\hline
8 &    2.43e-02 &  1.6  &   5.25e-02 &  1.6  &  1.99e-03  & 2.0  &   1.23e-02 &  1.9\\
16&    6.58e-03 &  1.9  &   1.46e-02 &  1.9  &  4.95e-04  & 2.0  &   3.12e-03 &  2.0\\
32&    1.68e-03 &  2.0  &   3.76e-03 &  2.0  &  1.24e-04  & 2.0  &   7.82e-04 &  2.0\\
64&    4.23e-04 &  2.0  &   9.49e-04 &  2.0  &  3.09e-05  & 2.0  &   1.96e-04 &  2.0\\
\hline
$h^{-1}$ & \multicolumn{4}{| c |}{$\kappa=20.$}& \multicolumn{4}{c}{\multirow{6}{*}{}}\\
\cline{1-5}
& $\|u_h-u \|_{0,h}$ & $r=$ & $\|u_h-u\|_{1,h}$ & $r=$ \\
\cline{1-5}
  8  & 4.64e-03 &  2.0  &   1.39e-02 &  1.8\\
  16 & 1.16e-03 &  2.0  &   3.57e-03 &  2.0\\
  32 & 2.91e-04 &  2.0  &   8.98e-04 &  2.0\\
  64 & 7.28e-05 &  2.0  &   2.25e-04 &  2.0\\
\cline{1-5}
\end{tabular}
\end{center}
}
\end{table}

Figure \ref{fig.DMP3} illustrates the numerical solutions and their maximum values on the domain boundary for the test problem \eqref{eq.testcase.3} on the uniform square partitions of size $16\times16$. The plots indicate that the maximum value of the numerical solution is not attained on the domain boundary, but satisfies the discrete maximum principle of $\max_{(x,y) \in \Omega_h}u_b \leq \max_{(x,y) \in \partial \Omega_h}\max(u_b,0)$ for all three cases of $\kappa =0.7$, $4.0$, and $20.0$.

\begin{figure}[h!]
\includegraphics [width=0.32\textwidth]{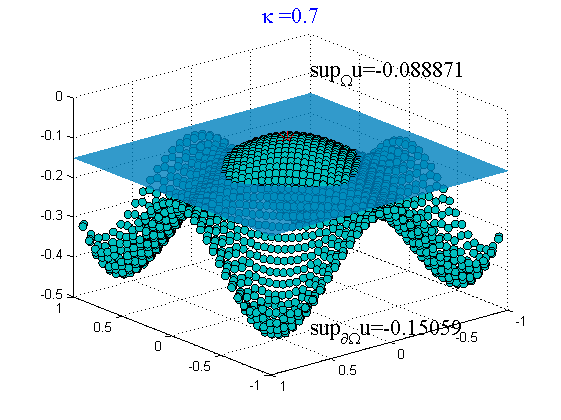}
\includegraphics [width=0.32\textwidth]{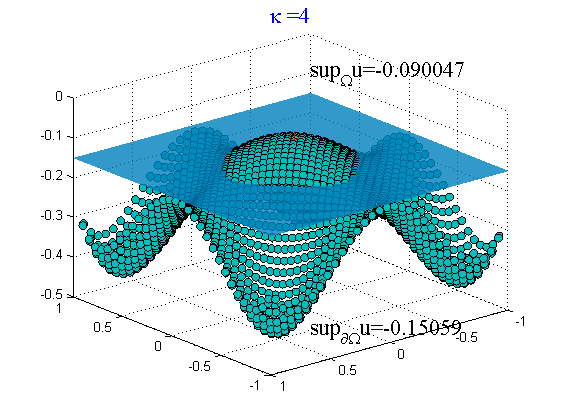}
\includegraphics [width=0.32\textwidth]{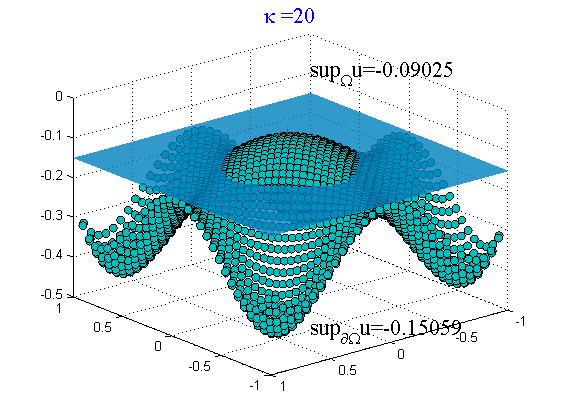}
\caption{\label{fig.DMP3} Plot of numerical solutions and their maximum values on the domain boundary arising from the SWG scheme \eqref{equation.SWG} for the test problem \eqref{eq.testcase.3}, using the uniform rectangular partition of size $16\times16$.}
\end{figure}

In summary, the discrete maximum principle is satisfied for all three test problems for the case of $\kappa=0.7$ and $\kappa=4$. The result is in good consistency with the DMP theory developed in Theorem \ref{Th.DMP.squaremesh}. The computational result also indicates a satisfaction of the DMP for the case of
$\kappa=20$ for which no theory was known.

%%%%%%%%%%%%%%%%%%%%%%%%%%%%%%%%%%%%%%%%%%%%%%%%%%%%%%%%%%%%%%%%%%%%%%%%%%%%%%%%%%%%%%%%%%%%%%%%%%%%%%%%%%%%%%%%%%%%%%%%%%%%%%%%%%%%%%%%%%%%%%%%%%%%%
\subsection{Numerical results for the finite difference scheme}
The goal of this section is to investigate the performance of the finite difference scheme \eqref{equation.SWG}.
For simplicity, we consider the model problem  \eqref{ellipticbdy}-\eqref{ellipticbc} with $\alpha=1$, $\bbeta=0$, and $c=0$ on the unit square domain $\Omega=(0,1)^2$. The two test cases in our numerical experiments assume the following exact solution and the right-hand side function:
%\begin{equation}
\begin{eqnarray}
&&\left\{
\begin{array}{lllll}
u=-x(x-1)y(y-1), \\
f=2x(x-1)+2y(y-1),
\end{array}
\right. \label{eq.testDMP.1}\\
~\nonumber\\
&&\left\{
\begin{array}{lllll}
u=-\sin(x)\sin(y)-x^2+y^2,\\
f=-2\sin(x)\sin(y).
\end{array}
\right. \label{eq.testDMP.2}
%\end{equation}
\end{eqnarray}
It is easy to see that the function $f$ is non-positve in both tests.

Table \ref{tab.testcaseFD12} show the error and convergence performance of the finite difference scheme \eqref{equation.SWG}. Optimal order of convergence can be seen for the $L^2$ error, while a superconvergence of order $\O(h^2)$ is illustrated for the numerical solution in the discrete $H^1$ norm.

\begin{table}[ht]
{\small
\begin{center}
\caption{Error and convergence performance of the finite difference scheme \eqref{equation.SWG} for the test cases \eqref{eq.testDMP.1} and \eqref{eq.testDMP.2}.}\label{tab.testcaseFD12}
\begin{tabular}{|c|cc|cc|cc|cc|}
\hline
     & \multicolumn{4}{| c |}{Test case \eqref{eq.testDMP.1}, $\kappa=4.0$} & \multicolumn{4}{| c |}{Test case \eqref{eq.testDMP.2}, $\kappa=4.0$}\\
\hline
$h^{-1}$ & $\|u_h-u \|_{0,h}$ & $r=$ & $\|u_h-u\|_{1,h}$ & $r=$ & $\|u_h-u \|_{0,h}$ & $r=$ & $\|u_h-u\|_{1,h}$ & $r=$\\
\hline
 8  &    4.59e-04 &     -   &  1.46e-03  &    -&    3.47e-05  &    -  &   4.10e-04 &     -\\
 16 &    1.14e-04 &  2.0   &  3.66e-04  & 2.0&    8.63e-06  & 2.0  &   1.02e-04 &  2.0\\
 32 &    2.85e-05 &  2.0   &  9.15e-05  & 2.0&    2.15e-06  & 2.0  &   2.56e-05 &  2.0\\
 64 &    7.12e-06 &  2.0   &  2.29e-05  & 2.0&    5.33e-07  & 2.0  &   6.42e-06 &  2.0\\
 128&    1.78e-06 &  2.0   &  5.72e-06  & 2.0&    1.30e-07  & 2.0  &   1.61e-06 &  2.0\\
\hline
\end{tabular}
\end{center}
}
\end{table}

Tables \ref{table.testDMP.12}-\ref{table.testDMP.12_20} show  $\displaystyle\max_{(x,y)\in\Omega_h}u_b(x,y)$ and $\displaystyle\max_{(x,y)\in\partial\Omega_h}u_b(x,y)$ for the numerical solutions arising from the finite difference scheme \eqref{equ.scheme-DF-SWG}. Table \ref{table.testDMP.12} is concerned with the case of the scheme with $\kappa = 0.7$ and $\kappa=4$ for which the DMP theory is applicable. It can be seen from the table that $\displaystyle \max_{(x,y)\in\Omega_h}u_b(x,y) < \max_{(x,y)\in\partial\Omega_h}u_b(x,y)$ so that the (strong) DMP is verified.
For curiosity, we also tested the case of $\kappa =20$ for which no DMP theory is known. It is interesting to note that the discrete maximum principle still holds true for the case of $\kappa=20$, as shown in Table \ref{table.testDMP.12_20}.

We point out that a penalization method was employed to implement the Dirichlet boundary condition in our computation so that the numerical boundary values are slightly different from the exact boundary data.

\begin{table}[H]
\begin{center}
\caption{DMP verification of the finite difference scheme \eqref{equ.scheme-DF-SWG} for the test problem \eqref{eq.testDMP.1} with $\kappa=0.7$ and $\kappa=4$ on uniform partitions.}\label{table.testDMP.12}
\begin{tabular}{||c||c|c|c||c|c|c||}
\hline
\multicolumn{7}{| c |}{Test problem \eqref{eq.testDMP.1}}\\
\hline
& \multicolumn{3}{| c |}{$\kappa=0.7$ }& \multicolumn{3}{| c |}{$\kappa=4.0$ }\\
\hline
$h$ & 1/8 &1/32 & 1/128  & 1/8 &1/32 & 1/128 \\
\hline
&&&&&&\\[-7pt]
$\displaystyle\max_{\partial\Omega}u_b(x,y)$  & -7.5e-11  & -4.9e-12  & -2.4e-13 & -6.5e-11  & -4.6e-12 &  -3.0e-13 \\
\hline
&&&&&&\\[-7pt]
$\displaystyle\max_{\Omega}u_b(x,y)$    & -7.5e-03  & -4.9e-04   & -3.0e-05 & -6.5e-03   & -4.6e-04  &  -3.0e-05 \\
\hline
\multicolumn{7}{| c |}{Test problem \eqref{eq.testDMP.2}}\\
\hline
& \multicolumn{3}{| c |}{$\kappa=0.7$ }& \multicolumn{3}{| c |}{$\kappa=4.0$ }\\
\hline
$h$ & 1/8 &1/32 & 1/128  & 1/8 &1/32 & 1/128\\
\hline
&&&&&&\\[-7pt]
$\displaystyle\max^{~}_{x\in\partial\Omega}u_b(x,y)$  &  0.9435  & 0.9866  & 0.9966 & 0.9435  & 0.9866 &  0.9966 \\
\hline
&&&&&&\\[-7pt]
$\displaystyle\max_{x\in\Omega}u_b(x,y)$          & 0.7448  &  0.9408   & 0.9855 &0.7627   & 0.9419  &   0.9855 \\
\hline
\end{tabular}
\end{center}
\end{table}

\begin{table}[H]
\begin{center}
\caption{DMP verification of the finite difference scheme \eqref{equ.scheme-DF-SWG} for the test problems \eqref{eq.testDMP.1} and \eqref{eq.testDMP.2} with $\kappa=20.0$ on uniform meshes.}\label{table.testDMP.12_20}
\begin{tabular}{||c||c|c|c||c|c|c||}
\hline
& \multicolumn{3}{| c |}{Test problem \eqref{eq.testDMP.1}, $\kappa=20.0$ }& \multicolumn{3}{| c |}{Test problem \eqref{eq.testDMP.2}, $\kappa=20.0$ }\\
\hline
$h$ & 1/8 &1/32 & 1/128  & 1/8 &1/32 & 1/128\\
\hline
&&&&&&\\[-7pt]
$\displaystyle\max_{\partial\Omega}u_b(x,y)$  &  -6.3e-11~ & -4.6e-12~& -3.0e-13~ & 0.9435~ & 0.9866~& 0.9966~ \\
\hline
&&&&&&\\[-7pt]
$\displaystyle\max_{\Omega}u_b(x,y)$    &  -6.3e-03~ & -4.6e-04~ & -3.0e-05~ & 0.7682~ & 0.9423~ & 0.9856~\\
\hline
\end{tabular}
\end{center}
\end{table}

\begin{figure}[h!]
\centering
\includegraphics [width=0.32\textwidth]{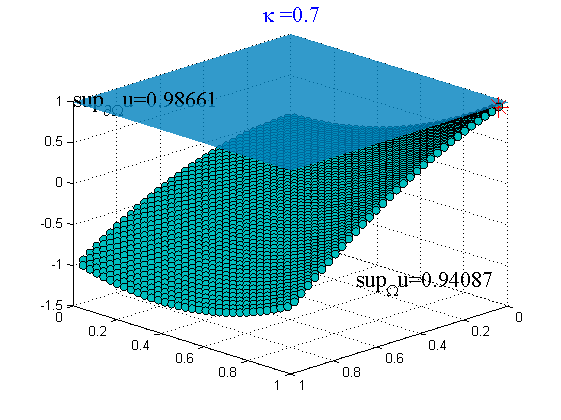}
\includegraphics [width=0.32\textwidth]{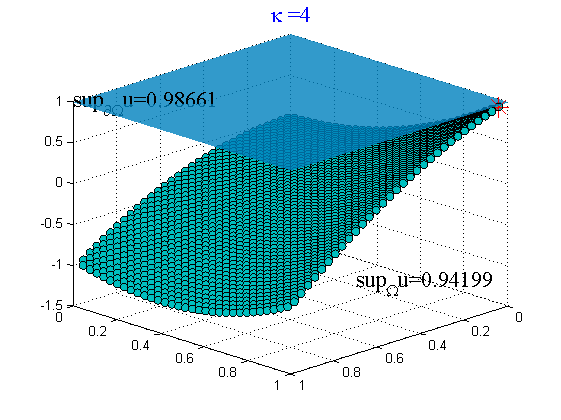}
\includegraphics [width=0.32\textwidth]{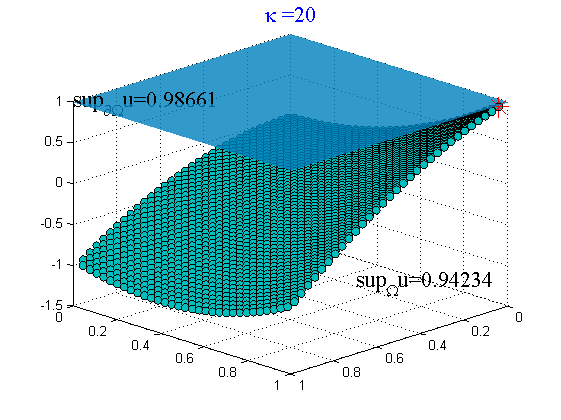}
\caption{\label{fig.DMP.1} DMP verification of the finite difference scheme \eqref{equ.scheme-DF-SWG} for the test problem \eqref{eq.testDMP.2}, using the uniform partition of size $32\times32$.}
\end{figure}

\end{document}